\documentclass{amsart}
\usepackage{amsmath,amsfonts,amsthm,amssymb,framed,enumerate}
\usepackage{tikz}
\usetikzlibrary{arrows,decorations.pathmorphing,backgrounds,positioning,fit}
\usepackage{setspace}
\newtheorem{thm}{Theorem}[subsection]

\newtheorem{lemma}[thm]{Lemma}

\newtheorem{coro}[thm]{Corollary}
\newtheorem{prop}[thm]{Proposition}
\theoremstyle{definition}
\newtheorem{defn}[thm]{Definition}
\theoremstyle{remark}
\newtheorem{rem}[thm]{Remark}

\numberwithin{equation}{section}

\def\ZZ{\mathbb{Z}}
\def\QQ{\mathbb{Q}}

\newcommand{\gmat}[2][ccccccccccccccccccccccccccccccccc]{\left[\begin{array}{#1} #2\\ \end{array}\right]}

\def\A{\mathcal{A}}
\def\U{\mathcal{U}}
\def\R{\mathcal{R}}

\def\x{\mathbf{x}}
\def\y{\mathbf{y}}

\def\codim{\text{codim}}
\def\Q{\mathsf{Q}}
\def\S{\mathcal{S}}

\def\K{\mathcal{F}}
\def\Den{\Gamma}
\def\D{\mathbb{D}}

\def\sM{\mathsf{M}}
\def\sD{\mathsf{D}}
\def\sB{\mathsf{B}}
\def\Mat{\text{Mat}}

\def\nr{{n}}

\def\nm{{m}}

\tikzstyle{mutable}=[inner sep=0.5mm,circle,draw,minimum size=2mm]
\tikzstyle{frozen}=[inner sep=1mm,rectangle,draw]
\tikzstyle{marked}=[inner sep=0.5mm,circle,draw,fill=black!50]
\tikzstyle{outline}=[thick,line width=1.5mm,draw=black!10]

\title{Computing upper cluster algebras}

\author{Jacob Matherne}
\address{Department of Mathematics,
Louisiana State University, Baton Rouge, LA 70808, USA}
\email{jmath34@lsu.edu}

\author{Greg Muller}
\address{Department of Mathematics,
Louisiana State University, Baton Rouge, LA 70808, USA}
\email{gmuller@lsu.edu}
\thanks{$2010$ \emph{Mathematics Subject Classification.} Primary 13F60, Secondary 14Q99}
\thanks{The second author was supported by the VIGRE program at LSU, National Science Foundation grant DMS-0739382.}

\begin{document}

\begin{abstract}
This paper develops techniques for producing presentations of upper cluster algebras.  These techniques are suited to computer implementation, and will always succeed when the upper cluster algebra is totally coprime and finitely generated.  We include several examples of presentations produced by these methods.
\end{abstract}

\maketitle

\section{Introduction}

\subsection{Cluster algebras}

Many notable varieties have a \emph{cluster structure}, in the following sense.  They are equipped with distinguished regular functions called \emph{cluster variables}, which are grouped into \emph{clusters}, each of which form a transcendence basis for the field of rational functions.  Each cluster is endowed with \emph{mutation} rules for moving to other clusters, and in this way, every cluster can be reconstructed from any other cluster.  Excamples of this include semisimple Lie groups \cite{BFZ05}, Grassmannians \cite{Sco06}, partial flag varieties \cite{GLS08}, moduli of local systems \cite{FG06}, and others.

The obvious algebra to consider in this situation is the \emph{cluster algebra} $\A$, the ring generated by the cluster variables.\footnote{Technically, the construction of the cluster algebra used in this note includes the inverses to a finite set.}  However, from a geometric perspective, the more natural algebra to consider is the \emph{upper cluster algebra} $\U$, defined by intersecting certain Laurent rings (see Remark \ref{rem: uppergeom} for the explicit geometric interpretation).

The \emph{Laurent phenomenon} guarantees that $\A\subseteq\U$.  This can be strengthened to an equality $\A=\U$ in many of the geometric examples and simpler classes of cluster algebras (such as acyclic and locally acyclic cluster algebras \cite{BFZ05,MulLA}).  In most cases where $\A=\U$ is known, the structures and properties of the algebra $\A=\U$ are fairly well-understood; for example, \cite[Corollary 1.21]{BFZ05} presents an acyclic cluster algebra as a finitely generated complete intersection.%; this includes acyclic cluster algebras (Section \ref{section: acyclic}, locally acyclic cluster algebras

However, there are examples where $\A\subsetneq \U$; the standard counterexample is the \emph{Markov cluster algebra} (see Remark \ref{rem: Markov} for details).  In these examples, both $\A$ and $\U$ are more difficult to work with directly, and either can exhibit pathologies. For example, the Markov cluster algebra is non-Noetherian \cite{MulLA}, and Speyer recently produced a non-Noetherian upper cluster algebra \cite{Speyer}.

\subsection{Presenting upper cluster algebras}

Nevertheless, because of its geometric nature, the authors expect that an upper cluster algebra $\U$ is generally better behaved than its cluster algebra $\A$.  This is supported in the few concretely understood examples where $\A\subsetneq\U$; however, the scarcity of examples makes investigating $\U$ difficult.

The goal of this note is to alleviate this problem by developing techniques to produce explicit presentations of $\U$.  The main tool is the following lemma, which gives several computationally distinct criteria for when a Noetherian ring $\S$ is equal to $\U$.
\begin{lemma}%\label{lemma: criterion}
If $\A$ is a cluster algebra with deep ideal $\D$, and $\S$ is a Noetherian ring such that $\A\subseteq \S\subseteq \U$, then the following are equivalent.
\begin{enumerate}
	\item $\S=\U$.
	\item $\S$ is normal and $\codim(\S\D)\geq2$.
	\item $\S$ is $S2$ and $\codim(\S\D)\geq 2$.
	\item $Ext^1_{\S}(\S/\S\D,\S)=0$.
	\item $\S f=(\S f:(\S \D)^\infty)$, where $f:=x_1x_2...x_\nm$ for some cluster $\x=\{x_1,...,x_\nr\}$.
\end{enumerate}
If $\S f\neq(\S f:(\S \D)^\infty)$, then $(\S f:(\S \D)^\infty)f^{-1}$ contains elements of $\U$ not in $\S$.
\end{lemma}
\noindent Here, the \emph{deep ideal} $\D$ is the ideal in $\A$ defined by the products of the mutable cluster variables (see Section \ref{section: uppergeom}).

This lemma is constructive, in that a negative answer to condition $(5)$ explicitly provides new elements of $\U$.  Even without a clever guess for a generating set of $\U$, iteratively checking this criterion and adding new elements can produce a presentation for $\U$.  Speyer's example demonstrates that this algorithm cannot always work; however, if $\U$ is finite, this approach will always produce a generating set (Corollary \ref{coro: finitealgo}).

Naturally, we include several examples of these explicit presentations.  Sections \ref{section: examples1} and \ref{section: examples2} contain presentations for the upper cluster algebras of the seeds pictured in Figure \ref{fig: listofseeds}.

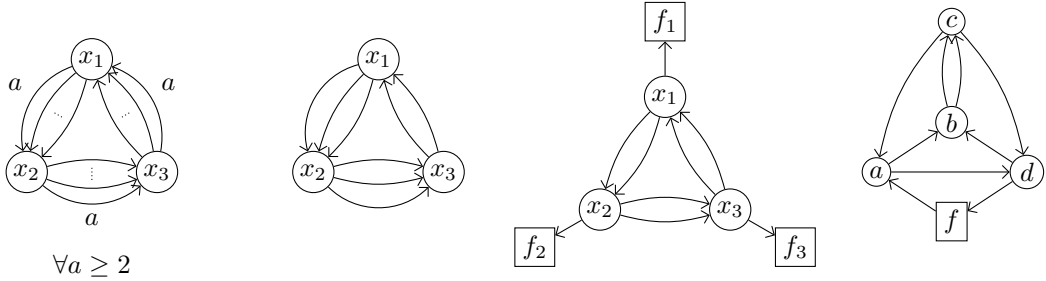
\begin{figure}[h!t]
	\begin{tikzpicture}
	\begin{scope}%[xshift=2.5in]
		\path[use as bounding box] (-1,-1) rectangle  (1,1.25);
		\node[mutable] (x) at (90:1) {$x_1$};
		\node[mutable] (y) at (210:1) {$x_2$};
		\node[mutable] (z) at (-30:1) {$x_3$};
%		\node[mutable] (x) at (90:1) {};
%		\node[mutable] (y) at (210:1) {};
%		\node[mutable] (z) at (-30:1) {};
		\draw[-angle 90,relative, out=15,in=165] (x) to node[inner sep=.1cm] (xy1m) {} (y);
		\draw[-angle 90,relative, out=-20,in=-160] (x) to node[inner sep=.1cm] (xy2m) {} (y);
		\draw[-angle 90,relative, out=-40,in=-140] (x) to node[above left] {$a$} (y); %node[above,rotate=60] {\tiny $a$-many}
		\draw[densely dotted] (xy1m) to (xy2m);
		\draw[-angle 90,relative, out=15,in=165] (y) to node[inner sep=.1cm] (yz1m) {} (z);
		\draw[-angle 90,relative, out=-20,in=-160] (y) to node[inner sep=.1cm] (yz2m) {} (z);
		\draw[-angle 90,relative, out=-40,in=-140] (y) to node[below] {$a$}  (z);
		\draw[densely dotted] (yz1m) to  (yz2m);
		\draw[-angle 90,relative, out=15,in=165] (z) to node[inner sep=.1cm] (zx1m) {} (x);
		\draw[-angle 90,relative, out=-20,in=-160] (z) to node[inner sep=.1cm] (zx2m) {} (x);
		\draw[-angle 90,relative, out=-40,in=-140] (z) to node[above right] {$a$}  (x);
		\draw[densely dotted] (zx1m) to  (zx2m);
		\node at (0,-1.75) {$\forall a\geq2$};
	\end{scope}
	\begin{scope}[xshift=1.5in]
		\node[mutable] (x) at (90:1) {$x_1$};
		\node[mutable] (y) at (210:1) {$x_2$};
		\node[mutable] (z) at (-30:1) {$x_3$};
%		\node[mutable] (x) at (90:1) {};
%		\node[mutable] (y) at (210:1) {};
%		\node[mutable] (z) at (-30:1) {};
		\draw[-angle 90,relative, out=15,in=165] (x) to (y);
		\draw[-angle 90,relative, out=-15,in=-165] (x) to (y);
		\draw[-angle 90,relative, out=-45,in=-495] (x) to (y);
		\draw[-angle 90,relative, out=15,in=165] (y) to (z);
		\draw[-angle 90,relative, out=-15,in=-165] (y) to (z);
		\draw[-angle 90,relative, out=-45,in=-495] (y) to (z);
		\draw[-angle 90,relative, out=15,in=165] (z) to (x);
		\draw[-angle 90,relative, out=-15,in=-165] (z) to (x);
	\end{scope}
	\begin{scope}[xshift=3in,yshift=-.5cm]
		\node[mutable] (x1) at (90:1) {$x_1$};
		\node[mutable] (x2) at (210:1) {$x_2$};
		\node[mutable] (x3) at (-30:1) {$x_3$};
        \node[frozen] (f1) at (90:2) {$f_1$};
        \node[frozen] (f2) at (210:2) {$f_2$};
        \node[frozen] (f3) at (-30:2) {$f_3$};
%		\node[mutable] (x1) at (90:1) {};
%		\node[mutable] (x2) at (210:1) {};
%		\node[mutable] (x3) at (-30:1) {};
%        \node[frozen] (f1) at (90:2) {};
%        \node[frozen] (f2) at (210:2) {};
%        \node[frozen] (f3) at (-30:2) {};
		\draw[-angle 90,relative, out=15,in=165] (x1) to (x2);
		\draw[-angle 90,relative, out=-15,in=-165] (x1) to (x2);
		\draw[-angle 90,relative, out=15,in=165] (x2) to (x3);
		\draw[-angle 90,relative, out=-15,in=-165] (x2) to (x3);
		\draw[-angle 90,relative, out=15,in=165] (x3) to (x1);
		\draw[-angle 90,relative, out=-15,in=-165] (x3) to (x1);
        \draw[-angle 90] (x1) to (f1);
        \draw[-angle 90] (x2) to (f2);
        \draw[-angle 90] (x3) to (f3);
	\end{scope}
	\begin{scope}[xshift=4.5in,yshift=-.5cm]
		\node[mutable] (a) at (-1,0) {$a$};
		\node[mutable] (b) at (0,.66) {$b$};
		\node[mutable] (c) at (0,2) {$c$};
		\node[mutable] (d) at (1,0) {$d$};
		\node[frozen] (f) at (0,-.66) {$f$};
%		\node[mutable] (a) at (-1,0) {};
%		\node[mutable] (b) at (0,.66) {};
%		\node[mutable] (c) at (0,2) {};
%		\node[mutable] (d) at (1,0) {};
%		\node[frozen] (f) at (0,-.66) {};
		\draw[-angle 90] (f) to (a);
		\draw[-angle 90] (a) to (b);
		\draw[-angle 90, out=105,in=255] (b) to (c);
		\draw[-angle 90,relative,out=-15,in=195] (c) to (a);
		\draw[-angle 90] (a) to (d);
		\draw[-angle 90] (d) to (b);
		\draw[-angle 90, out=75,in=285] (b) to (c);
		\draw[-angle 90,relative,out=15,in=165] (c) to (d);
		\draw[-angle 90] (d) to (f);
	\end{scope}	
	\end{tikzpicture}
\caption{Seeds of upper cluster algebras presented in this note.}
\label{fig: listofseeds}
\end{figure}

\begin{rem}
To compute examples, we use a variation of Lemma \ref{lemma: criterion} involving \emph{lower bounds} and \emph{upper bounds}, which requires that our cluster algebras are \emph{totally coprime}.
\end{rem}

%Many motivating examples of cluster algebras are found inside the rational functions on important spaces:
%
%
%In each case, special functions on these spaces are cluster variables (generalized minors, Pl\"ucker coordinates, and lambda lengths, respectively), and functions which occur together in a cluster have some compatibility property.
%
%Many of the most important examples of cluster algebras are found in the coordinate rings of
%
%Introduced in \cite{FZ02}, cluster algebras are commutative algebras with

\section{Cluster algebras}

\emph{Cluster algebras} are a class of commutative unital domains.  Up to a finite localization, they are generated in their field of fractions by distinguished elements, called \emph{cluster variables}. The cluster variables (and hence the cluster algebra) are produced by an recursive procedure, called \emph{mutation}.  While cluster algebras are geometrically motivated, their construction is combinatorial and determined by some simple data called a 'seed'.

\subsection{Ordered seeds}

A matrix $\sM\in \Mat_{m,m}(\ZZ)$ is \emph{skew-symmetrizable} if there is a non-negative, diagonal matrix $\sD\in \Mat_{m,m}(\ZZ)$ such that $\sD\sM$ is skew-symmetric; that is, that $(\sD\sM)^\top=-\sD\sM$.

Let $\nr \geq \nm\geq0$ be integers, and let $\sB\in \Mat_{\nr,\nm}(\ZZ)$ be an integer valued $\nr\times \nm$-matrix.  Let $\sB^0\in \Mat_{\nm,\nm}(\ZZ)$ be the \emph{principal part}, the submatrix of $\sB$ obtained by deleting the last $\nr-\nm$ rows.

An \textbf{ordered seed} is a pair $(\x,\sB)$ such that...
\begin{itemize}
	\item $\sB\in \Mat_{\nr,\nm}(\ZZ)$,
	\item $\sB^0$ is skew-symmetrizable, and
	\item $\x=(x_1,...,x_\nr)$ is an $\nr$-tuple of elements in a field $\K$ of characteristic zero, which is a free generating set for $\K$ as a field over $\QQ$.
\end{itemize}
The various parts of an ordered seed have their own names.
\begin{itemize}
	\item The matrix $\sB$ is the \emph{exchange matrix}.
	\item The $\nr$-tuple $\x$ is the \emph{cluster}.
	\item Elements $x_i\in \x$ are \emph{cluster variables}.\footnote{Many authors do not consider frozen variables to be cluster variables, instead referring to them as `geometric coefficients', following \cite{FZ07}.}  These are further subdivided by index.
	\begin{itemize}
		\item If $0< i\leq \nm$, $x_j$ is a \emph{mutable variable}.
		\item If $\nm< i\leq \nr$, $x_j$ is a \emph{frozen variable}.
	\end{itemize}
\end{itemize}

The ordering of the cluster variables in $\x$ is a matter of convenience.  A permutation of the cluster variables which preserves the flavor of the cluster variable (mutable/frozen) acts on the ordered seed by reordering $\x$ and conjugating $\sB$.

%When the principal part $\sB^0$ is skew-symmetric, an exchange matrix $\sB$ can be encoded in an \textbf{ice quiver} $\Q$. A \emph{quiver} is a finite directed graph, and an \emph{ice quiver} is a quiver together with a distinguished subset of \emph{frozen vertices}.  Given skew-symmetric $\sB$, the vertex set of $\Q$ is $\{1,...,\nr\}$, the frozen vertices are $\{\m+1,....,\nr\}$.  The number of arrows from $i$ to $j$ is $\sB_{ij}=-\sB_{ji}$, or $0$ if $i$ and $j$ are frozen.\footnote{If only one of $i$ and $j$ is frozen, then only one of $\sB_{ij}$ and $-\sB_{ji}$ exists; use that one.}

A skew-symmetric seed $(\x,\sB)$ can be diagrammatically encoded as an \emph{ice quiver} (Figure \ref{fig: seedquiver}).	Put each mutable variable $x_i$ in a circle, and put each frozen variable $x_i$ in a square. For each pair of indices $i<j$ with $i\leq \nm$, add $\sB_{ji}$ arrows from $i$ to $j$, where `negative arrows' go from $j$ to $i$.
%\begin{itemize}
%	\item Put each mutable $x_i$ in a circle.
%	\item Put each frozen $x_i$ in a square.
%	\item For each pair of indices $i<j$ with $i\leq \m$, add $\sB_{ji}$ arrows from $i$ to $j$, where `negative arrows' go from $j$ to $i$.
%\end{itemize}

\begin{figure}[h!t]
	\begin{tikzpicture}
%	\begin{scope}[xshift=-1.25in]
%		\node (array) at (0,0) {$\x=\{x_1,x_2,x_3\}$};
%	\end{scope}
	\begin{scope}
		\node (array) at (0,0) {$\left(\x=\{x_1,x_2,x_3\},\sB=\gmat{0 & -3  \\ 3 & 0 \\ -2 & 1 }\right)$};
	\end{scope}
	\begin{scope}[xshift=2.75in]
		\node (label) at (-2,0) {$\Q=$};
		\node[mutable] (x) at (90:1) {$x_1$};
		\node[mutable] (y) at (210:1) {$x_2$};
		\node[frozen] (z) at (-30:1) {$x_3$};
		\draw[-angle 90,relative, out=15,in=165] (x) to (y);
		\draw[-angle 90,relative, out=-15,in=-165] (x) to (y);
		\draw[-angle 90,relative, out=-45,in=-495] (x) to (y);
%		\draw[-angle 90,relative, out=15,in=165] (y) to (z);
		\draw[-angle 90,relative, out=-15,in=-165] (y) to (z);
%		\draw[-angle 90,relative, out=-45,in=-495] (y) to (z);
		\draw[-angle 90,relative, out=15,in=165] (z) to (x);
		\draw[-angle 90,relative, out=-15,in=-165] (z) to (x);
	\end{scope}
	\end{tikzpicture}
\caption{The ice quiver associated to a seed}
\label{fig: seedquiver}
\end{figure}
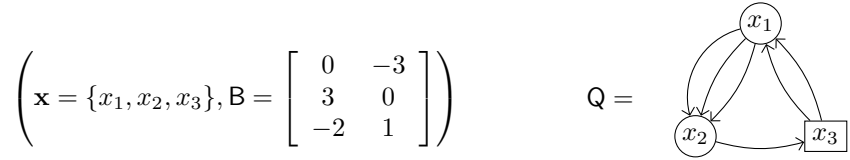

A seed $(\x,\sB)$ is called \emph{acyclic} if $\Q$ does not contain a directed cycle of mutable vertices.  The seed in Figure \ref{fig: seedquiver} is acyclic.

\subsection{Cluster algebras}

Given an ordered seed $(\x,\sB)$ and some $1\leq k\leq \nm$, define the \emph{mutation} of $(\x,\sB)$ at $k$ to be the ordered seed $(\x',\sB')$, where
\[ x_i':= \left\{\begin{array}{cc}
\left(\prod_{\sB_{jk}>0} x_j^{\sB_{jk}} + \prod_{\sB_{jk}<0} x_j^{-\sB_{jk}}\right)x_i^{-1} & \text{if $i=k$} \\
x_i & \text{otherwise}
\end{array}\right\}\]

\[ \sB_{ij}':= \left\{\begin{array}{cc}
-\sB_{ij} & \text{if $i=k$ or $j=k$} \\
\sB_{ij} + \frac{|\sB_{ik}|\sB_{kj}+ \sB_{ik}|\sB_{kj}|}{2} & \text{otherwise}
\end{array}\right\}\]
Since $(\x',\sB')$ is again an ordered seed, mutation may be iterated at any sequence of indices in $1,2,...,\nm$.  Mutation twice in a row at $k$ returns to the original ordered seed.

Two ordered seeds $(\x,\sB)$ and $(\y,\mathsf{C})$ are \emph{mutation-equivalent} if $(\y,\mathsf{C})$ is a obtained from $(\x,\sB)$ by a sequence of mutations and permutations.

\begin{defn}
Given an ordered seed $(\x,\sB)$, the associated \textbf{cluster algebra} $\A(\x,\sB)$ is the subring of the ambient field $\K$ generated by
\[\{x_i^{-1} \mid \nm<i\leq \nr\} \cup \bigcup_{(\mathbf{y},\mathsf{C})\sim (\x,\sB)}\mathbf{y}\] %as $(\mathbf{y},\mathsf{C})$ runs over all seeds mutation-equivalent to $(\x,\sB)$.
\end{defn}
%\begin{rem}
%If there are non-invertible frozen variables (ie, $\ni\neq\nr$), this is not a cluster algebra in the sense of \cite{???} or \cite{???}, because the `coefficients semigroup' $\ZZ^{\ni-m}\times \NN^{\nr-\ni}$ is not a group.
%\end{rem}

A \emph{cluster variable} in $\A(\x,\sB)$ is a cluster variable in any ordered seed mutation-equivalent to $(\x,\sB)$, and it is mutable or frozen based on its index in any seed. A \emph{cluster} in $\A(\x,\sB)$ is a set of cluster variables appearing as the cluster in some ordered seed.  Mutation-equivalent seeds define the same cluster algebra $\A$.  The seed will often be omitted from the notation when clear.

%\begin{rem}
%Some variants of cluster algebras do not include the inverses to some or all of the frozen variables.  The resulting algebras cover a wider range of examples, but the arguments developed in this note are not well-suited to this context.  The technical reason for this is that sufficiently simple covers (as in Section \ref{???}) can be difficult to come by, even in the case of acyclic cluster algebras.
%\end{rem}

A cluster algebra $\A$ is \textbf{acyclic} if there exists an acyclic seed of $\A$; usually, an acyclic cluster algebra will have many non-acyclic seeds as well.  Acyclic cluster algebras have proven to be the most easily studied class; for example, \cite[Corollary 1.21]{BFZ05} gives a presentation of $\A$ with $2\nr$ generators and $\nr$ relations.

\subsection{Upper cluster algebras}

A basic tool in the theory of cluster algebras is the following theorem, usually called the \emph{Laurent phenomenon}.

\begin{thm}[Theorem 3.1, \cite{FZ02}]
Let $\A$ be a cluster algebra, and $\x=\{x_1,x_2,...,x_{\nr}\}$ be a cluster in $\A$.  As subrings of $\K$,
\[ \A \subset \ZZ[x_1^{\pm1},...,x_\nr^{\pm1}]\]
This is the localization of $\A$ at the mutable variables $x_1,...,x_\nm$.
\end{thm}

%\noindent For $\nr=\ni$, this was proven in \cite{FZ01} and again in \cite{BFZ05}.  The stronger version appearing in \cite[Proposition 3.6]{FZ08} implies the general case.

%The theorem has the following concrete interpretation. For any cluster $\x$, any element $f\in \K$ can be written as a rational polynomial $f(\x)$ in the variables $x_1,...,x_\nr$.  When $f\in \A\subset \K$, then the rational polynomial $f(\x)$ is a Laurent polynomial in $\ZZ[x_1^{\pm1},...,x_\nr^{\pm1}]$.

The theorem says elements of $\A$ can be expressed as Laurent polynomials in many different sets of variables (one such expression for each cluster).  The set of all rational functions in $\K$ with this property is an important algebra in its own right, and the central object of study in this note.

\begin{defn}
Given a cluster algebra $\A$, the \textbf{upper cluster algebra} $\U$ is defined
\[ \U := \bigcap_{\substack{\text{clusters} \\ \x=\{x_1,...,x_\nr\} \text{ in }\A }} \ZZ[x_1^{\pm1},...,x_\nr^{\pm1}] \subset \K\]
\end{defn}

The Laurent phenomenon is equivalent to the containment $\A\subseteq \U$.
%Despite its definition, the upper cluster algebra $\U$ is more natural than $\A$ from a geometric perspective (see Remark \ref{rem: uppergeom}).  %Because of this, it is expected that $\U$ is generally better behaved than $\A$; the following proposition demonstrates one such property.

\begin{prop}\cite[Proposition 2.1]{MulLA}\label{prop: Unormal}
Upper cluster algebras are normal.
\end{prop}
\begin{rem}\label{rem: internormal}
Any intersection of normal domains in a fraction field is normal.
\end{rem}

\subsection{Lower and upper bounds}\label{section: bounds}

The cluster algebras that have finitely many clusters have an elegant classification by Dynkin diagrams \cite{FZ03}.  However, such \emph{finite-type} cluster algebras are quite rare; even the motivating examples are frequently \emph{infinite-type}.  Working with infinite-type $\A$ or $\U$ can be daunting because the definitions involve infinite generating sets or intersections (this is especially a problem for computer computations).

Following \cite{BFZ05}, to any seed $\x$, we associate bounded analogs of $\A$ and $\U$ called \emph{lower} and \emph{upper bounds}.  The definitions are the same, except the only seeds considered are $\x$ and those seeds a single mutation away from $\x$.

As a standard abuse of notation, for a fixed seed $(\x=\{x_1,x_2,...,x_{\nr}\},\sB)$, let $x_i'$ denote the mutation of $x_i$ in $(\x,\sB)$.

\begin{defn} Let $(\x,\sB)$ be a seed in $\K$.

The \textbf{lower bound} $L_{\x}$ is the subring of $\K$ generated by $\{x_1,x_2,...,x_{\nr}\}$, the one-step mutations $\{x_1',x_2',...,x_{\nm}'\}$, and the inverses to invertible frozen variables $\{x_{\nm+1}^{-1},...,x_{\nr}^{-1}\}$.

The \textbf{upper bound} $\U_{\x}$ is the intersection in $\K$ of the $n+1$ Laurent rings corresponding to $\mathbf{x}$ and its one-step mutations.
\[ \U_{\mathbf{x}} := \mathbb{Z}[x_1^{\pm1},...,x_{\nr}^{\pm1}]\cap \bigcap_i \mathbb{Z}[x_1^{\pm1},...,x_{i-1}^{\pm1},x_i'^{\pm1},x_{i+1}^{\pm1}...,x_{\nr}^{\pm1}]\]
\end{defn}

The names `lower bound' and `upper bound' are justified by the obvious inclusions
\[ L_\x\subseteq \A\subseteq \U\subseteq \U_\x\]
%In \cite{BFZ05}, the authors studied when these inclusions were equalities.\footnote{However, the results of \cite{BFZ05} only apply when $\ni=\nr$; ie, when all frozen variables are invertible.}  Since this note is most interested in $\U$, we review their results on when $\U=\U_\x$.

When does $\U=\U_\x$?
A seed $(\x,\sB)$ is \emph{coprime} if every pair of columns in $\sB$ are linearly independent.
A cluster algebra is \textbf{totally coprime} if every seed is coprime.
\begin{thm}[Corollary 1.7, \cite{BFZ05}]\label{thm: tpupper}
If $\A$ is totally coprime, then $\U=\U_\x$ for any seed $(\x,\sB)$.
\end{thm}

Mutating a seed can make coprime seeds non-coprime (and vice versa), so verifying a cluster algebra is totally coprime may be hard in general.  A stronger condition is that the exchange matrix $\sB$ has \emph{full rank} (ie, kernel $0$); this is preserved by mutation, so it implies the cluster algebra $\A(\sB)$ is totally coprime.

\begin{thm}[Proposition 1.8, \cite{BFZ05}]\label{thm: fullrank}
If the exchange matrix $\sB$ of a seed of $\A$ is full rank, then $\A$ is totally coprime.
\end{thm}
Of course, there are many totally coprime cluster algebras which are not full rank.\footnote{Proposition \ref{prop: 3totallycoprime} provides a class of such examples.}

%\begin{prop}
%Let $\A$ be a cluster algebra with 3 mutable variables.  Either $\A$ is totally coprime, or $\A$ has a seed $(\x,\sB)$ with $1\leq i\leq j\leq 3$ such that $\sB_{ij}=0$.
%\end{prop}
%\begin{proof}
%Let $(\x,\sB)$ in $\A$. For any $1\leq i\leq j\leq 3$, if $\sB_{ij}\neq0$, then columns $i$ and $j$ in $\sB$ are linearly independent, because $\sB_{ii}=0$.  Either there are some $i,j$ with $\sB_{ij}=0$, or $(\x,\sB)$ is a coprime seed.
%\end{proof}
%\begin{rem}
%In the latter case of the proposition, $(\x,\sB)$ is an `acyclic' seed \cite[Definition 1.14]{BFZ05}.  In this case, $\A$ and $\U$ are well understood by the methods of \cite{BFZ05}.
%\end{rem}

\section{Regular functions on an open subscheme}

This section collects some generalities about the ring we denote $\Den(R,I)$ \--- the ring of regular functions on the open subscheme of $Spec(R)$ whose complement is $V(I)$ \--- and relates this idea to cluster algebras.%\footnote{The underlying geometric motivation for $\Den(R,I)$ will not be needed, however.}

\subsection{Definition}

Let $R$ be a domain with fraction field $\K(R)$.\footnote{All rings in this note are commutative and unital, but need not be Noetherian.}  For any ideal $I\subset R$, define the ring $\Den(R,I)$ as the intersection (taken in $\K(R)$)
\[ \Den(R,I) := \bigcap_{r\in I} R[r^{-1}]\]
\begin{rem}
In geometric terms, $\Den(R,I)$ is the ring of rational functions on $Spec(R)$ which are regular on the complement of $V(I)$. As a consequence, $\Den(R,I)$ only depends on $I$ up to radical.  Neither of these facts are necessary for the rest of this note, however.
\end{rem}

\begin{prop}\label{prop: dengens}
If $I$ is generated by a set $\pi\subset R$, then
\[ \Den(R,I) = \bigcap_{r\in \pi} R[r^{-1}]\]
\end{prop}
\begin{proof}
Choose some $f\in I$, and write $f=\sum_{r\in \pi_0} b_r r$, where $\pi_0$ is a finite subset of $\pi$.  Let $g\in \bigcap_{r\in \pi} R[r^{-1}]$; therefore, there are $n_r\in R$ and $\alpha_r\in \mathbb{N}$ such that $g=\frac{n_r}{r^{\alpha_r}}$ for all $r\in \pi$.  Define
\[ \beta = 1+\sum_{r\in \pi_0}\alpha_r\]
and consider $f^\beta g$.  Expanding $f^\beta = \left(\sum b_rr\right)^\beta$, every monomial expression in the $\{r\}$ contains at least one $r'\in \pi_0$ with exponent greater or equal to $\alpha_{r'}$.  Since $r'^{\alpha_{r'}} g=n_{r'}\in R$, it follows that $f^\beta g\in R$ and $g\in R[f^{-1}]$.
Therefore, $\bigcap_{r\in \pi}R[r^{-1}]\subseteq \Den(R,I)$.
%Since $\Den(R,I)$ is an intersection over a larger set, the other containment is clear.
\end{proof}

%\begin{prop}
%If $S$ is a ring $R\subseteq S\subseteq \K(R)$, then
%\[S \cdot \Den(R,I) \subseteq \Den(S,SI)\]
%If $I$ is a finitely generated $R$-ideal, then this is equality.
%\end{prop}
%\begin{proof}
%If $\pi$ generates $I$ over $R$, then $\pi$ generates $SI$ over $S$.  Then
%\[ S\cdot \Den(R,I) = S\cdot \bigcap_{r\in \pi} R[r^{-1}] \subseteq \bigcap_{r\in\pi}S[r^{-1}] = \Den(S,SI)\]
%Assume $\pi$ is finite.  Let $f\in \bigcap_{r\in \pi}S[r^{-1}]$, and choose some $r\in \pi$.  Since $S\subseteq \K(R)$, there are $a_r,b_r\in R$ and $n_r\in \mathbb{N}$ such that $f=\frac{a_r}{b_r r^{n_r}}$.  Letting $B=\prod_{s\in \pi} b_s$, then
%\[ B f = B\frac{a_r}{b_r r^{n_r}}=\frac{a_r}{r^{n_r}} \left(\prod_{s\in \pi, s\neq r} b_s\right)\in R[r^{-1}]\]
%Since $r$ was arbitrary in $\pi$, $Bf\in \Den(R,I)$.  Since $B^{-1}\in S$, $f\in S\cdot \Den(R,I)$.
%\end{proof}
%
%\begin{rem}
%I do not know what happens for infinitely generated $I$.  A counterexample would have to be an ideal $I$ which is big enough such that every generating set has no common multiple, but small enough that $\Reg(R,I)\neq R$.  I believe this is closely related to the fact that depth behaves poorly in the non-Noetherian world.
%\end{rem}

\begin{prop}\label{prop: denintermediate}
If $R\subseteq S\subseteq \Den(R,I)$, then $\Den(R,I)=\Den(S,SI)$.
\end{prop}
\begin{proof}
For $i\in I$, $\Den(R,I)\subset R[r^{-1}]$, and so $S\subset R[r^{-1}]$.  Then $S[r^{-1}]=R[r^{-1}]$ for all $r\in I$.  If $\pi$ generates $I$ over $R$, then $\pi$ generates $SI$ over $S$.  By Proposition \ref{prop: dengens},
\[ \Den(R,I) = \bigcap_{r\in \pi} R[r^{-1}] = \bigcap_{r\in\pi}S[r^{-1}] = \Den(S,SI)\]
This completes the proof.
\end{proof}

\subsection{Upper cluster algebras}\label{section: uppergeom}

The relation between a cluster algebra $\A$ and its upper cluster algebra $\U$ is an example of this construction.
%Let $\A$ be a cluster algebra with upper cluster algebra $\U$.
Define the \textbf{deep ideal} $\D$ of $\A$ by
\[ \D := \sum_{\text{clusters }\{x_1,x_2,...,x_{\nr}\}} \A x_1x_2...x_{\nm}\]
That is, it is the $\A$-ideal generated by the product of the mutables variables in each cluster.
\begin{prop}\label{prop: Uasden}
$\Den(\A,\D)=\U$.
\end{prop}
\begin{proof}Since $\D$ is generated by the products of the mutable variables in the clusters,
\begin{eqnarray*}
\Den(\A,\D) &=& \bigcap_{\text{clusters }\{x_1,x_2,...,x_n\}} \A [(x_1x_2...x_{\nm})^{-1}]\\
 &=& \bigcap_{\text{clusters }\{x_1,x_2,...,x_n\}} \mathbb{Z}[x_1^{\pm1},x_2^{\pm1},...,x_{\nr}^{\pm1}]
\end{eqnarray*}
Thus, $\Den(\A,\D)=\U$.
\end{proof}

\begin{rem} \label{rem: uppergeom}
The proposition is equivalent to the following well-known geometric interpretation of $\U$. % which has appeared in the literature \cite{???}.
If $\{x_1,...,x_{\nr}\}$ is a cluster, then the isomorphism
\[\A [(x_1x_2...x_{\nm})^{-1}]\simeq \mathbb{Z}[x_1^{\pm1},x_2^{\pm1},...,x_{\nr}^{\pm1}]\]
determines an open inclusion $\mathbb{G}^{\nr}_{\ZZ}\hookrightarrow Spec(\A)$.\footnote{These open algebraic tori are called \emph{toric charts} in \cite{Sco06} and \emph{cluster tori} in \cite{MulLA}.}  The union of all such open affine subschemes is a smooth open subscheme in $Spec(\A)$, whose complement is $V(\D)$.\footnote{This union is called the \emph{cluster manifold} in \cite{GSV03}.}  The proposition states that $\U$ is the ring of regular functions on this union.
\end{rem}

\subsection{Upper bounds}

Let $(\x,\sB)$ be a seed with $\x=\{x_1,x_2,...,x_n\}$.  As in Section \ref{section: bounds}, let $x_i'$ denote the mutation of $x_i$ in $\x$.  The \textbf{lower deep ideal} $\D_\mathbf{x}$ is the $L_{\x}$-ideal
\[ \D_\x := L_\x (x_1x_2...x_{\nm}) + \sum_i L_\x (x_1x_2...x_{i-1}x_i'x_{i+1}...x_{\nm})\]

Proposition \ref{prop: Uasden} has an analog.
\begin{prop}\label{prop: Uxasden}
$\Den(L_\x,\D_\x)=\U_\x$
\end{prop}
\begin{proof}
Since $\D_\x$ is generated by the products of the mutable variables in $L_\x$,
\begin{eqnarray*}
\Den(L_\x,\D_\x) &=& L_\x[(x_1x_2...x_{\nm})^{-1}]\cap \bigcap_i L_\x[(x_1x_2...x_i'...x_{\nm})^{-1}] \\
 &=& \mathbb{Z}[x_1^{\pm1},...,x_{\nr}^{\pm1}]\cap \bigcap_i \mathbb{Z}[x_1^{\pm1},...,x_i'^{\pm1},...,x_{\nr}^{\pm1}]
\end{eqnarray*}
Thus, $\Den(L_\x,\D_\x)=\U_\x$.
\end{proof}

In practice, $\Den(L_\x,\D_\x)$ is much easier to work with than $\Den(\A,\D)$, because the objects involved are defined by finite generating sets.

\begin{rem}
For any set of clusters $S$ in $\A$, one may define $L_S$ generated by the variables in $S$, $\U_S$ as the intersection of the Laurent rings of clusters in $S$, and $\D_S$ an ideal in $L_S$ generated by the products of clusters in $S$.  Again, one has  $\U_S = \Den(L_S,\D_S)$.
\end{rem}

\section{Criteria for $\Den(R,I)$}

Given a `guess' for $\Den(R,I)$ \--- a ring $S$ such that $R\subseteq S\subseteq \Den(R,I)$ \--- there are several criteria for verifying if $S=\Den(R,I)$. This section develops these criteria.

%This section develops techniques for computing an explicit presentations of $\Den(R,I)$, with the aim of computing presentations of $\U_\x$ and $\U$.

\subsection{Saturations}

Given two ideals $I,J$ in $R$, define the \textbf{saturation}
\[ (J:I^\infty) = \{r\in R \mid \forall g\in I,\; \exists n\in \mathbb{N} \text{ s.t. } rg^n\in J\}\]
Computer algebra programs can compute saturations, at least when $R$ is finitely generated.

\begin{rem}
%\textbf{Caution!}
When $I$ is not finitely generated, this definition of saturation may differ from the infinite union $\bigcup_n(J:I^n)$, which amounts to reversing the order of quantifiers.
\end{rem}

Saturations can be used to compute the sub-$R$-module of $\Den(R,I)$ with denominator $f$.

\begin{prop}\label{prop: densat}
If $f\in R$, then
\[ Rf^{-1}\cap \Den(R,I)=(Rf:I^\infty)f^{-1}\]
\end{prop}
\begin{proof}
If $g\in R\cap f\Den(R,I)$, then for any $r\in I$, we may write $gf^{-1}=hr^{-m}$ for some $h\in R$ and $m\in \mathbb{N}$.  Then $gr^{m}=hf\in Rf$; and so $g\in (Rf:I^\infty)$.

If $g\in (Rf:I^\infty)$, then for any $r\in I$, there is some $m$ such that $gr^{m}\in Rf$.  It follows that $gf^{-1}\in Rr^{-m}\subset R[r^{-1}]$.  Therefore, $gf^{-1}\in \Den(R,I)$, and so $g\in f\Den(R,I)$.
\end{proof}

%\begin{coro}\label{coro: densat}
%If $f\in I$, then
%\[ \Den(R,I) = \bigcup_n (Rf^n:I^\infty)f^{-n}\]
%\end{coro}

Saturations can also detect when $R=\Den(R,I)$.

\begin{prop}\label{prop: dencheck}
Let $f\in I$. Then $R=\Den(R,I)$ if and only if
$Rf = (Rf:I^\infty)$.
\end{prop}
\begin{proof}
If $R=\Den(R,I)$, then $(Rf:I^\infty)=Rf\cap \Den(R,I)=Rf$.

Assume $Rf = (Rf:I^\infty)$. Let $g\in \Den(R,I)$, and let $n$ be the smallest integer such that $f^ng\in R$.  If $n\geq1$, then
\[f(f^{n-1}g)\in R\cap (f\Den(R,I)) =(Rf:I^\infty)=Rf\]
and so $f^{n-1}g\in R$, contradicting minimality of $n$.  So $g\in R$, and so $\Den(R,I)=R$.
\end{proof}

\subsection{The saturation criterion} %Given a presentation for $R$ and $I$, how can one compute a presentation $\Den(R,I)$, either by hand or with a computer?  %One basic computational obstacle is that $\Den(R,I)$ need not be finitely generated, even if $R$ is.  %This means that any algorithm which attempts to produce a presentation for $\Den(R,I)$ must be open-ended.
%We sketch an iterative algorithm which attempts to present $\Den(R,I)$ by producing successively larger subrings.

%Assume that one has has a `guess' for $\Den(R,I)$: a ring $S$ such that $R\subseteq S\subseteq \Den(R,I)$.  If $R$ and $I$ are finitely generated, then checking these containments is feasible.

Given a ring $S$ with $R\subseteq S\subseteq \Den(R,I)$, the following lemma gives a necessary and computable criterion for when $S=\Den(R,I)$.  Perhaps more importantly, if $S\subsetneq \Den(R,I)$, it explicitly gives new elements of $\Den(R,I)$, which can be used to generate a better guess $S'\subseteq \Den(R,I)$.
%
% and if not, it explicitly gives new elements in $\Den(R,I)$.  In the latter case, these new elements can be used to generate a better guess $S\subsetneq S'\subseteq \Den(R,I)$.

\begin{lemma}\label{lemma: satcriterion}
Let $R\subseteq S\subseteq \Den(R,I)$.  For and $f\in I$,
\[ S\subseteq (S f:(S I)^\infty)f^{-1}\subset \Den(R,I)\]
Furthermore, either
\begin{itemize}
	\item $S=\Den(R,I)$, or
	\item $S\subsetneq (S f:(S I)^\infty)f^{-1}\subseteq \Den(R,I)$.
\end{itemize}
\end{lemma}
\begin{proof}
By Proposition \ref{prop: denintermediate}, $\Den(R,I)=\Den(S,SI)$.  The containment $(S f:(S I)^\infty)f^{-1}\subset \Den(R,I)$ follows from Proposition \ref{prop: densat}.  The containment $S\subseteq (S f:(S I)^\infty)f^{-1}$ is clear from the definition of the saturation.
If $(S f:(S I)^\infty)f^{-1}=S$, then Proposition \ref{prop: dencheck} implies that $S=\Den(R,I)$.
\end{proof}

\subsection{Noetherian algebraic criteria}

When the ring $S$ is Noetherian, there are several alternative criteria to verify that $S=\Den(R,I)$.\footnote{However, even when $R$ is Noetherian, one cannot always expect that $\Den(R,I)$ is Noetherian.}
When $S$ is also normal, these criteria are sharp, but none of them can give a constructive negative answer similar to Lemma \ref{lemma: satcriterion}.

The definitions of `codimension', `S2' and `depth' used here are found in \cite{Eis95}.
\begin{lemma}\label{lemma: Hartog}
Let $R\subseteq S\subseteq \Den(R,I)$.
If $S$ is Noetherian, then each of the following statements implies the next.
\begin{enumerate}
	\item $S$ is normal and $\codim(SI)\geq2$.
	\item $S$ is $S2$ and $\codim (SI)\geq2$.
	\item $\text{depth}_S(SI)\geq2$; that is, $Ext^1_S(S/SI,S)=0$.
	\item $S=\Den(R,I)$.
\end{enumerate}
If $S$ is normal and Noetherian, then the above statements are equivalent.
\end{lemma}
\begin{proof}
$(1) \Rightarrow (2)$. By Serre's criterion \cite[Theorem 11.5.i]{Eis95}, a normal Noetherian domain is S2.

$(2) \Rightarrow (3)$. The S2 condition implies that every ideal of codimension $\geq2$ has depth $\geq2$; see the proof of \cite[Theorem 18.15]{Eis95}.\footnote{Some sources take this as the definition of S2.}

$(\text{Not }4) \Rightarrow (\text{Not }3)$. Assume that $S\subsetneq \Den(R,I)$, and let $f\in I$.  By Lemma \ref{lemma: satcriterion} and Proposition \ref{prop: densat},
\[S\subsetneq (Sf:(SI)^\infty)f^{-1} = Sf^{-1}\cap \Den(S,SI)\]
Since $S$ is Noetherian, $SI$ is finitely-generated, and so it is possible to find an element $g\in Sf^{-1}\cap \Den(S,SI)$ such that $g\not\in S$ but $gI\subseteq S$.
%Choose any $gf^{-1}\in Sf^{-1}\cap \Den(S,SI)$ but not in $S$.
The natural short exact sequence
%Let $M$ be the cokernel of this inclusion.  Then
\[ 0\rightarrow S\hookrightarrow Sg\rightarrow Sg/S\rightarrow0\]
is an essential extension,
and so $Ext^1_S(Sg/S,S)\neq 0$.

The map $S/SI\rightarrow Sg/S$ which sends $1$ to $g$ is a surjection, and its kernel $K$ is a torsion $S$-module.  Hence, there is a long exact sequence which contains  %The left-exact functor $Hom_S(-,S)$ to $K\rightarrow S/SI\rightarrow Sg/S$ gives a long exact sequence which contains
\[ \dots \rightarrow Hom_S(K,S)\rightarrow Ext^1_S(Sg/S,S)\rightarrow Ext^1_S(S/SI,S)\rightarrow...\]
Since $K$ is torsion, $Hom_S(K,S)=0$, and so $Ext^1_S(S/SI,S)\neq0$.

$(S \text{ normal})+(\text{Not }1) \Rightarrow (\text{Not }4)$.  Assume that $S$ is normal, and that $\codim(SI)=1$.  Therefore, there is a prime $S$-ideal $P$ containing $SI$ with $\codim(P)=1$.
By Serre's criterion \cite[Theorem 11.5.ii]{Eis95},  the localization $S_P$ is a discrete valuation ring.  Let $\nu:\K(S)^*\rightarrow \ZZ$ be the corresponding valuation.

Let $a_1,a_2,...,a_j$ generate $P$ over $S$.  Then $a_1,a_2,...,a_j$ generate $S_PP$ over $S_P$.  There must be some $a_i$ with $\nu(a_i)=1$, and this element generates $S_PP$.  Reindexing as needed, assume that $\nu(a_1)=1$.
For each $a_i$, there exists $f_i,g_i\in S-P$ such that
\[ a_i = \frac{f_i}{g_i}a_1^{\nu(a_i)}\]
Let $d=gcd(\nu(a_i))$.  Then, for all $1\leq k\leq j$,
\[ x:= \frac{1}{a_1^d}\left(\prod_{1<i\leq j} g_i^{\frac{d}{\nu(a_i)}}\right) = \left(\frac{f_k}{a_k^d}\right)^{\frac{d}{\nu(a_k)}}\left(\prod_{\stackrel{1<i\leq j}{ i\neq k}} g_i^{\frac{d}{\nu(a_i)}}\right) \in S[a_k^{-1}]\]
It follows that $x\in \Den(S,P)\subseteq \Den(S,SI) =\Den(R,I)$.  However, since $\nu(x)=-d$, it follows that $x\not\in S$, and so $S\neq \Den(R,I)$.
\end{proof}

\begin{rem}
The implication $(1)\Rightarrow (4)$ is one form of the `algebraic Hartog lemma', in analogy with Hartog's lemma in complex analysis.
\end{rem}

\begin{rem}
The assumption that $S$ is Noetherian is essential.  If
\[R=S=\mathbb{C}[[x^t \mid t\in \QQ_{\geq0}]]\]%_{t\in \QQ_{\geq0}}\]
is the ring of Puiseux series without denominator, and $I$ is generated by $\{x^t\}_{t>0}$, then
$R$ is normal and $Ext^1(R/I,R)=0$. %\footnote{The `right' definition of S2 and depth in a non-Noetherian ring is far from clear.} 
Nevertheless,
\[\Den(R,I) = \mathbb{C}[[x^t \mid t\in \QQ]] \neq R\]
is the field of all Puiseux series.
\end{rem}

\subsection{Criteria for $\U$}

We restate the previous criteria for upper cluster algebras.

%The following lemma gives four computable conditions for when a guess $\S$ is $\U$, the last of which is constructive, in that it produces new elements of $\U$ if $\S\neq \U$.

\begin{lemma}\label{lemma: criterion}
If $\A$ is a cluster algebra with deep ideal $\D$, and $\S$ is a Noetherian ring such that $\A\subseteq \S\subseteq \U$, then the following are equivalent.
\begin{enumerate}
	\item $\S=\U$.
	\item $\S$ is normal and $\codim(\S\D)\geq2$.
	\item $\S$ is $S2$ and $\codim(\S\D)\geq 2$.
	\item $Ext^1_{\S}(\S/\S\D,\S)=0$.
	\item $\S f=(\S f:(\S \D)^\infty)$, where $f:=x_1x_2...x_\nm$ for some cluster $\x=\{x_1,...,x_\nr\}$.
\end{enumerate}
If $\S f\neq(\S f:(\S \D)^\infty)$, then $(\S f:(\S \D)^\infty)f^{-1}$ contains elements of $\U$ not in $\S$.
\end{lemma}
%\begin{proof}
%$(2)\Rightarrow (3)\Rightarrow (4)\Rightarrow (1)$ is Lemma \ref{lemma: Hartog}.
%
%$(1) \Rightarrow (2)+(3)+(4)$. If $\S=\U$, then $\S$ is normal by Proposition \ref{prop: Unormal}, and  Lemma \ref{lemma: Hartog} implies statements $(2)$, $(3)$ and $(4)$.
%
%$(1)\Leftrightarrow (5)$.  Since $f=x_1x_2...x_n\in \D$ by construction, this is a case of Lemma \ref{lemma: satcriterion}.
%\end{proof}

However, we are interested in infinite-type cluster algebras, where the containments $\A\subseteq\S\subseteq \U$ cannot be naively verified by hand or computer.  This is where lower and upper bounds are helpful, since the analogous containments can be checked directly.

\begin{lemma}\label{lemma: tccriterion}
If $(\x,\sB)$ is seed in a totally coprime cluster algebra $\A$ and $\S$ a Noetherian ring such that $L_\x\subseteq \S\subseteq \U_\x$, then the following are equivalent.
\begin{enumerate}
	\item $\S=\U_\x=\U$.
	\item $\S$ is normal and $\codim(\S\D_\x)\geq2$.
	\item $\S$ is $S2$ and $\codim(\S\D_\x)\geq 2$.
	\item $Ext^1_{\S}(\S/\S\D_\x,\S)=0$.
	\item $\S f=(\S f:(\S \D_\x)^\infty)$, where $f$ is the product of the mutable variables in $\x$.
\end{enumerate}
If $\S f\neq(\S f:(\S \D_\x)^\infty)$, then $(\S f:(\S \D_\x)^\infty)f^{-1}$ contains elements of $\U$ not in $\S$.
\end{lemma}

Note that $\U_\x$ is normal by Remark \ref{rem: internormal}, and so the strong form of Lemma \ref{lemma: Hartog} applies.
%\begin{proof}
%The
%\end{proof}
\begin{rem}
Criterion $(2)$ was used implicitly in the proofs of \cite[Theorem 2.10]{BFZ05} and \cite[Proposition 7]{Sco06}, and a form of it is stated in \cite[Proposition 3.6]{FP}.
\end{rem}

%We feel compelled to mention the following corollary.
%\begin{coro}
%If $\A$ is finitely generated, then $\U$ is the normalization of $\A$.
%\end{coro}
%\begin{proof}
%
%\end{proof}

\section{Presenting $\U$}

This section outlines the steps for checking if a set of Laurent polynomials generates a totally coprime upper cluster algebra $\U$ over the subring of frozen variables.  %, and if not, how to iterate the process

\subsection{From conjectural generators to a presentation}

\def\ZP{\mathbb{Z}\mathbb{P}}

Fix a seed $(\x=\{x_1,...,x_\nm \},\sB)$ in a totally coprime cluster algebra $\A$.  Let
\[\ZP:= \ZZ[x_{\nm+1}^{\pm1},x_{\nm+2}^{\pm1},...,x_{\nr}^{\pm1}]\]
be the \emph{coefficient ring} \--- the Laurent ring generated by the frozen variables and their inverses.

Start with a finite set of Laurent polynomials in $\ZZ[x_1^{\pm1},...,x_{\nr}^{\pm1}]$, which hopefully generates $\U$ over $\ZP$. We assume that all the initial mutable variables $x_1,...,x_{\nr}$ are in this set. %\footnote{In lieu of a better guess, $x_1,...x_n,x_1',...,x_n'$ is an acceptable first guess.}
Write this set as
\[ x_1,x_2,...x_{\nm},y_1,...., y_p\]
where
\[ y_i=\frac{N_i(x_1,...,x_n)}{x_1^{\alpha_{1i}}x_2^{\alpha_{2i}}...x_{\nr}^{\alpha_{\nr i}}}\in
\ZZ[x_1^{\pm1},...,x_{\nr}^{\pm1}]\]
for some polynomial $N_i(x_1,...,x_n)$.

\begin{itemize}
	\item \textbf{Compute the ideal of relations.} Let
	\[\widetilde{\S}:=\ZP[x_1,...,x_\nm,y_1,...,y_p] \]
	be a polynomial ring over $\ZP$ (here, the $y_i$s are just symbols).  Define $\widetilde{I}$ to be the $\widetilde{\S}$-ideal generated by elements of the form
	\[ y_i(x_1^{\alpha_{1i}}x_2^{\alpha_{2i}}...x_{\nr}^{\alpha_{\nr i}}) - N_i(x_1,...,x_\nr)\]
	as $i$ runs from $1$ to $p$.  Let $I:= (\widetilde{I}:\widetilde{\S}(x_1...x_{\nm})^\infty)$ be the saturation of $I$ by the principal $\widetilde{\S}$-ideal generated by the product of the mutable variables $x_1x_2...x_{\nm}$.
	
	\begin{lemma}\label{lemma: presentation}
		The sub-$\ZP$-algebra of $\ZZ[x_1^{\pm1},...,x_{\nr}^{\pm1}]$ generated by
		\[ x_1,x_2,...x_{\nm},y_1,...., y_p\]
		is naturally isomorphic to the quotient $\S:=\widetilde{\S}/I$.
	\end{lemma}
	\begin{proof}
	Let the localization $\widetilde{\S}[(x_1x_2...x_\nm)^{-1}]$ is the ring \[\ZZ[x_1^{\pm1},...,x_{\nr}^{\pm1},y_1,...,y_p]\]
    The induced ideal $\widetilde{\S}[(x_1x_2...x_\nm)^{-1}]\widetilde{I}$ is generated by elements of the form
	\[ y_i - (x_1^{-\alpha_{1i}}x_2^{-\alpha_{2i}}...x_{\nr}^{-\alpha_{\nr i}})N_i(x_1,...,x_\nr)\]
	and so the quotient $\widetilde{\S}[(x_1x_2...x_\nm)^{-1}]/\widetilde{\S}[(x_1x_2...x_\nm)^{-1}]\widetilde{I}$ eliminates the $y_i$s and is isomorphic to $\ZZ[x_1^{\pm1},...,x_{\nr}^{\pm1}]$.  The kernel of the composition
	\[ \widetilde{\S}\rightarrow\widetilde{\S}[(x_1x_2...x_\nm)^{-1}]\rightarrow \ZZ[x_1^{\pm1},...,x_{\nr}^{\pm1}]\]
	consists of elements $r\in \widetilde{\S}$ such that $(x_1x_2...x_{\nm})^ir\in I$ for some $i$; this is the saturation $I$.
	\end{proof}

    \item \textbf{Verify that $L_\x\subseteq \S\subseteq \U_\x$.}  For the first containment, it suffices to check that $x_1',x_2',...,x_{\nm}'\in \S$, because the other generators of $L_\x$ are in $\S$ by construction.
    For the second containment, it suffices to check that for each $1\leq i\leq \nm$ and $1\leq k\leq p$,
        \[ y_k\in \ZZ[x_1^{\pm1},...,x_i'^{\pm1},...,x_\nr^{\pm1}]\]
        This is because $x_1,...,x_{\nm},x_{\nm+1}^{\pm1},...,x_{\nr}^{\pm1}$ are in $\U_\x$ by the Laurent phenomenon.

    \item \textbf{Check whether $\S=\U$ using Lemma \ref{lemma: tccriterion}.}  Any of the four criteria $(2)-(5)$ in Lemma \ref{lemma: tccriterion} can be used.  They all may be implemented by a computer, and each method potentially involves a different algorithm, so any of the four might be the most efficient computationally.

    \item \textbf{If $\S\subsetneq \U$, find additional generators and return to the beginning.}  If $\S\neq \U$, then $(\S f: (\S\D_\x)^\infty)f^{-1}$ contains elements of $\U$ which are not in $\S$ (where $f=x_1x_2...x_\nm$).  One or more of these elements may be added to the original list of Laurent polynomials to get a larger guess $S'$ for $\U$.  Note that any $S'$ produced this way satisfies $L_\x\subseteq S'\subseteq \U_\x$.

\end{itemize}

\subsection{An iterative algorithm}\label{section: algorithm}

The preceeding steps can be regarded as an iterative algorithm for producing successively larger subrings $S\subseteq \U$ as follows.  Start with an initial guess $L_\x\subseteq S\subseteq \U_\x$.  In lieu of cleverness, the lower bound $L_\x=S$ makes an functional initial guess; this amounts to starting with generators $x_1,...,x_\nm,x_1',...,x_\nm'$.

Denote $S_1:=S$, and inductively define $S_{i+1}$ to be the sub-$\ZP$-algebra of $\QQ(x_1,x_2,...,x_\nr)$ generated by $S_i$ and $(S_i f:(S_i I)^\infty)f^{-1}$.  If $S_i$ is finitely generated over $\ZP$ (resp. Noetherian), then the saturation $(S_i f:(S_iI)^\infty)$ is finitely generated over $S_i$ and so $S_{i+1}$ is finitely generated over $\ZP$ (resp. Noetherian).

This gives a nested sequence of subrings
\[ L_\x\subseteq S=S_1\subseteq S_2\subseteq S_3\subseteq ...\subseteq \U=\U_\x\]
By  Lemma \ref{lemma: satcriterion}, if $S_i=S_{i+1}$, then $S_i=S_{i+1}=S_{i+2}=...=\U=\U_\x$.

%One obstacle is that $\Den(R,I)$ need not be a finitely generated over $S$, even if $S$ is finitely generated \cite{???}.  Since each $S_i$ must be finitely generated over $S$ (if $S$ is finitely generated), this algorithm need not terminate.  The following proposition demonstrates that this is the only obstacle.
\begin{prop}
If $\U$ is finitely generated over $S$, then for some $i$, $S_i=\U$.
\end{prop}
\begin{proof}
Let $f=x_1x_2...x_\nm$.  By Proposition \ref{prop: densat},
\[(S_if:(S_i \D_\x)^\infty) = S_if^{-1} \cap \U\]
Induction on $i$ shows that
$Sf^{-i}\cap \U \subseteq S_{i+1}$.
If $\U$ is finitely generated over $S$, then there is some $i$ such that $Sf^{-i+1}$ contains a generating set, and so $S_{i}=\U$.
\end{proof}

\begin{coro}\label{coro: finitealgo}
Let be $\A$ a totally coprime cluster algebra, and $S=L_\x$ for some seed in $\A$.  If $\U$ is finitely generated, then $\U=S_i$ for some $i$.
\end{coro}
\noindent In other words, this algorithm will always produce $\U$ in finitely many steps, even starting with the `worst' guess $L_\x$.

\begin{rem}
This algorithm can be implemented by computational algebra software, so long as the initial guess $S$ is finitely presented.  However, in the authors' experience, naively implementing this algorithm was computationally prohibitive after the first step.  A more effective approach was to pick a few simple elements of $(S_i f:(S_iI)^\infty)$ and use them to generate a bigger ring $S_{i+1}$.
\end{rem}

\section{Examples: $\nm=\nr=3$}\label{section: examples1}

The smallest non-acyclic seed will have $\nm=\nr=3$; that is, 3 mutable variables and no frozen variables.  We consider these examples.

%If the number of mutable variables $\nm\geq2$, a seed is acyclic.  Since our interest is in non-acyclic examples, the first examples of such seeds will have $\nm=\ni=\nr$.
%
%Every seed of a cluster algebra with 2 or fewer mutable variables is acyclic, and so \ref{prop: acyclicpres} gives a presentation of $\A=\U=\U_\x$.  The first interesting examples are cluster algebras with three mutable variables (ie, $\nm=3)$.

\subsection{Generalities}

Consider an arbitrary skew-symmetric seed $(\x,\sB_{a,b,c})$ with $\nm=\nr=3$, as in Figure \ref{fig: rank3}.  Let $\A_{a,b,c}$ and $\U_{a,b,c}$ be the corresponding cluster algebra and upper cluster algebra, respectively.\footnote{The notation $\U_{a,b,c}$ is dangerous, in that it leaves no room to distinguish between the upper cluster algebra and the upper bound of $\sB_{a,b,c}$.  However, we will only consider non-acyclic examples, and so by Theorem \ref{thm: tpupper}, these two algebras coincide.  The reader is nevertheless warned.
}

\begin{figure}[h!t]	
	\begin{tikzpicture}
	\begin{scope}
		\node (array) at (0,0) {$\sB_{a,b,c}=\gmat{0 & -a & c \\ a & 0 & -b \\ -c & b & 0 }$};
	\end{scope}
	\begin{scope}[xshift=2.5in]
		\path[use as bounding box] (-1,-1) rectangle  (1,1);
		\node (label) at (-2,0) {$\Q_{a,b,c}=$};
		\node[mutable] (x) at (90:1) {$x_1$};
		\node[mutable] (y) at (210:1) {$x_2$};
		\node[mutable] (z) at (-30:1) {$x_3$};
		\draw[-angle 90,relative, out=15,in=165] (x) to node[inner sep=.1cm] (xy1m) {} (y);
		\draw[-angle 90,relative, out=-20,in=-160] (x) to node[inner sep=.1cm] (xy2m) {} (y);
		\draw[-angle 90,relative, out=-40,in=-140] (x) to node[above left] {$a$} (y); %node[above,rotate=60] {\tiny $a$-many}
		\draw[densely dotted] (xy1m) to (xy2m);
		\draw[-angle 90,relative, out=15,in=165] (y) to node[inner sep=.1cm] (yz1m) {} (z);
		\draw[-angle 90,relative, out=-20,in=-160] (y) to node[inner sep=.1cm] (yz2m) {} (z);
		\draw[-angle 90,relative, out=-40,in=-140] (y) to node[below] {$b$}  (z);
		\draw[densely dotted] (yz1m) to  (yz2m);
		\draw[-angle 90,relative, out=15,in=165] (z) to node[inner sep=.1cm] (zx1m) {} (x);
		\draw[-angle 90,relative, out=-20,in=-160] (z) to node[inner sep=.1cm] (zx2m) {} (x);
		\draw[-angle 90,relative, out=-40,in=-140] (z) to node[above right] {$c$}  (x);
		\draw[densely dotted] (zx1m) to  (zx2m);
	\end{scope}
	\end{tikzpicture}
\caption{A general skew-symmetric seed with 3 mutable variables}
\label{fig: rank3}
\end{figure}
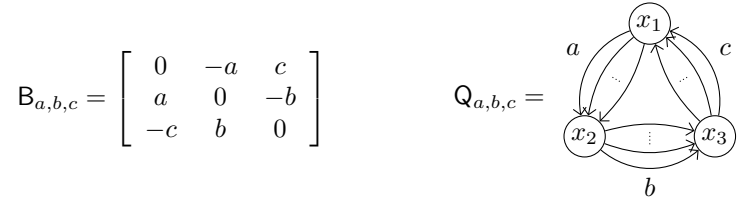

%\[ \sB_{a,b,c}=\gmat{ 0 & a & -c \\ -a & 0 & b \\ c & -b & 0 }\]
%for some integers $a,b,c\in \ZZ$.

The seed $(\x,\sB_{a,b,c})$ is acyclic unless $a,b,c>0$ or $a,b,c<0$, and permuting the variables can exchange these two inequalities.  Even when $a,b,c>0$, the cluster algebra $\A_{a,b,c}$ may not be acyclic, since there may be a acyclic seed mutation equivalent to $(\x,\sB_{a,b,c})$.  Thankfully, there is a simple inequality which detects when $\A_{a,b,c}$ is acyclic.
\begin{thm}\cite[Theorem 1.1]{BBH11}\label{thm: nonacyclicrank3}
Let $a,b,c>0$.  The seed $(\x,\sB_{a,b,c})$ is mutation-equivalent to an acyclic seed if and only if $a<2$, $b<2$, $c<2$, or
\[ abc-a^2-b^2-c^2+4< 0\]
\end{thm}

Acyclic $\A_{a,b,c}=\U_{a,b,c}$ can be presented using \cite[Corollary 1.21]{BFZ05}; and so we focus on the non-acyclic cases.  As the next proposition shows, these cluster algebras are totally coprime, and so it will suffice to present $\U_\x$.

\begin{prop}\label{prop: 3totallycoprime}
Let $\A$ be a cluster algebra with $\nm=3$.  If $\A$ is not acyclic, then $\A$ is totally coprime.
\end{prop}
\begin{proof}
Let $(\x,\sB)$ be a non-acyclic seed for $\A$ with quiver $\Q$; that is, there is a directed cycle of mutable cluster variables.  There are no 2-cycles in $\Q$ by construction, and so the directed cycle in $\Q$ passes through every vertex.  It follows that $\sB_{ij}\neq0$ if $i\neq j$.  Then the $i$th and $j$th columns are linearly independent, because $\sB_{ii}=0$ and $\sB_{ij}\neq0$.  Hence, $(\x,\sB)$ is a coprime seed, and $\A$ is totally coprime.
%
%For any $1\leq i\leq j\leq 3$, if $\sB_{ij}\neq0$, then columns $i$ and $j$ in $\sB$ are linearly independent, because $\sB_{ii}=0$.  Either there are some $i,j$ with $\sB_{ij}=0$, or $(\x,\sB)$ is a coprime seed.
\end{proof}

\begin{rem}
This proof does not assume that $\sB^0$ is skew-symmetric or that $\nr=3$ (ie, that there are no frozen variables).
\end{rem}

%\begin{rem}
%In the latter case of the proposition, $(\x,\sB)$ is an `acyclic' seed \cite[Definition 1.14]{BFZ05}.  In this case, $\A$ and $\U$ are well understood by the methods of \cite{BFZ05}.
%\end{rem}

%The algebra $\A_{a,b,c}$ also has a natural grading.  Given a seed $(\x,\sB_{a,b,c})$, define degrees
%\[ \deg(x_1) = b,\;\;\; \deg(x_2) = c,\;\;\; \deg(x_3) = a\]
%In terms of the quiver associated to $\sB_{a,b,c}$, the degree of a cluster variable is the number of arrows which do not touch that variable.
%
%\begin{prop}
%This gives $\A_{a,b,c}$ and $\U_{a,b,c}$ a $\ZZ$-grading, which does not depend on the choice of initial seed.
%\end{prop}
%\begin{proof}
%
%\end{proof}
%
%\begin{coro}
%$\A$ is non-acyclic if and only if the cluster variables all have degree $>0$, or all have degree $<0$.
%\end{coro}
%
%Therefore, if $a,b,c>0$, then a non-acyclic $\A$ is $\mathbb{N}$-graded with degree zero part $0$.
%
%\begin{rem}
%In general, cluster algebras are graded by $\ZZ^{\nr}/Im(\sB)$, where $\deg(x_i)$ is the class of $e_i$ in the quotient.\footnote{This is essentially due to \cite{GSV???}.} For $a,b,c$ not all zero, the rank of $\sB_{a,b,c}=2$, and so $\A(\x,\sB_{a,b,c})$ has a unique $\ZZ$-grading (up to scaling) compatible with this general grading.
%\end{rem}

\subsection{The $(a,a,a)$ cluster algebra}\label{section: aaa}

%Let $a\geq0$, and let $\U_{a,a,a}$ be the upper cluster algebra of the seed $(\x,\sB_{a,a,a})$ (Figure \ref{fig: gMarkov}).
Consider $a=b=c\geq0$ as in Figure \ref{fig: aaa}.

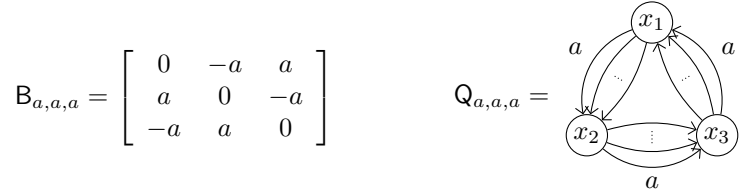
\begin{figure}[h!t]	
	\begin{tikzpicture}
	\begin{scope}
		\node (array) at (0,0) {$\sB_{a,a,a}=\gmat{0 & -a & a \\ a & 0 & -a \\ -a & a & 0 }$};
	\end{scope}
	\begin{scope}[xshift=2.5in]
		\path[use as bounding box] (-1,-1) rectangle  (1,1.25);
		\node (label) at (-2,0) {$\Q_{a,a,a}=$};
		\node[mutable] (x) at (90:1) {$x_1$};
		\node[mutable] (y) at (210:1) {$x_2$};
		\node[mutable] (z) at (-30:1) {$x_3$};
		\draw[-angle 90,relative, out=15,in=165] (x) to node[inner sep=.1cm] (xy1m) {} (y);
		\draw[-angle 90,relative, out=-20,in=-160] (x) to node[inner sep=.1cm] (xy2m) {} (y);
		\draw[-angle 90,relative, out=-40,in=-140] (x) to node[above left] {$a$} (y); %node[above,rotate=60] {\tiny $a$-many}
		\draw[densely dotted] (xy1m) to (xy2m);
		\draw[-angle 90,relative, out=15,in=165] (y) to node[inner sep=.1cm] (yz1m) {} (z);
		\draw[-angle 90,relative, out=-20,in=-160] (y) to node[inner sep=.1cm] (yz2m) {} (z);
		\draw[-angle 90,relative, out=-40,in=-140] (y) to node[below] {$a$}  (z);
		\draw[densely dotted] (yz1m) to  (yz2m);
		\draw[-angle 90,relative, out=15,in=165] (z) to node[inner sep=.1cm] (zx1m) {} (x);
		\draw[-angle 90,relative, out=-20,in=-160] (z) to node[inner sep=.1cm] (zx2m) {} (x);
		\draw[-angle 90,relative, out=-40,in=-140] (z) to node[above right] {$a$}  (x);
		\draw[densely dotted] (zx1m) to  (zx2m);
	\end{scope}
%	\begin{scope}[xshift=2.5in]
%		\node (label) at (-2,0) {$\Q_{a,a,a}=$};
%		\node[mutable] (x) at (90:1) {$x_1$};
%		\node[mutable] (y) at (210:1) {$x_2$};
%		\node[mutable] (z) at (-30:1) {$x_3$};
%		\draw[-angle 90,relative, out=15,in=165] (x) to node[inner sep=.1cm] (xy1m) {} (y);
%		\draw[-angle 90,relative, out=-20,in=-160] (x) to node[inner sep=.1cm] (xy2m) {} (y);
%		\draw[-angle 90,relative, out=-40,in=-140] (x) to (y);
%        \draw[densely dotted] (xy1m) to  (xy2m);
%		\draw[-angle 90,relative, out=15,in=165] (y) to node[inner sep=.1cm] (yz1m) {} (z);
%		\draw[-angle 90,relative, out=-20,in=-160] (y) to node[inner sep=.1cm] (yz2m) {} (z);
%		\draw[-angle 90,relative, out=-40,in=-140] (y) to (z);
%        \draw[densely dotted] (yz1m) to  (yz2m);
%		\draw[-angle 90,relative, out=15,in=165] (z) to node[inner sep=.1cm] (zx1m) {} (x);
%		\draw[-angle 90,relative, out=-20,in=-160] (z) to node[inner sep=.1cm] (zx2m) {} (x);
%		\draw[-angle 90,relative, out=-40,in=-140] (z) to (x);
%        \draw[densely dotted] (zx1m) to  (zx2m);
%	\end{scope}
	\end{tikzpicture}
\caption{The exchange matrix and quiver for the $(a,a,a)$ cluster algebra}
\label{fig: aaa}
\end{figure}

If $a=0$ or $1$, then $\A_{a,a,a}$ is acyclic.\footnote{In fact, finite-type of type $A_1\times A_1\times A_1$ or $A_3$, respectively.}  For $a\geq2$, $\A_{a,a,a}$ is not acyclic by Theorem \ref{thm: nonacyclicrank3}.

\begin{rem}\label{rem: Markov}
The case $a=2$ was specifically investigated in \cite{BFZ05}, as the first example of a cluster algebra for which $\A\neq\U$, and it has been subsequently connected to the Teichm\"uller space of the the once-punctured torus and to the theory of Markov triples \cite[Appendix B]{FG07} ($\A_{2,2,2}$ is sometimes called the \emph{Markov cluster algebra}).  See Section \ref{section: MwP} for the analog of $\U_{2,2,2}$ with a specific choice of frozen variables.
\end{rem}

\begin{prop}
For $a\geq2$, the upper cluster algebra $\U_{a,a,a}$ is generated over $\ZZ$ by
\[ x_1,x_2,x_3, M:= \frac{x_1^a+x_2^a+x_3^a}{x_1x_2x_3}\]
The ideal of relations among these generators is generated by
\[  x_1x_2x_3M-x_1^a-x_2^a-x_3^a=0\]
\end{prop}
\begin{proof}
Since $a^3-3a^2+4\geq0$ for $a\geq2$, Theorem \ref{thm: nonacyclicrank3} implies that this cluster algebra is not acyclic, and Proposition \ref{prop: 3totallycoprime} implies that it is totally coprime.

The element $x_1x_2x_3M-x_1^a-x_2^a-x_3^a$ in $\ZZ[x_1,x_2,x_3,M]$ is irreducible. The ideal it generates is prime and therefore it is saturated with respect to $x_1x_2x_3$.  By Lemma \ref{lemma: presentation},
\[ \S = \ZZ[x_1,x_2,x_3,M]/<x_1x_2x_3M-x_1^a-x_2^a-x_3^a>\]
is the subring of $\ZZ[x_1^{\pm1},x_2^{\pm1},x_3^{\pm1}]$ generated by $x_1,x_2,x_3$ and $M$.

The following identities imply that $L_\x\subset \S$.
\[ x_1' = \frac{x_2^a+x_3^a}{x_1} = x_2x_3M-x_1^a,\;\;\;
 x_2' = \frac{x_1^a+x_3^a}{x_2} = x_1x_3M-x_2^a,\;\;\;
  x_3' = \frac{x_1^a+x_2^a}{x_3} = x_1x_2M-x_3^a\]
The following identities imply that $\S\subset \U_\x$.
\[ M= \frac{x_1'^{a+1}+(x_2^a+x_3^a)^a}{x_1'^ax_2x_3}= \frac{x_2'^{a+1}+(x_1^a+x_3^a)^a}{x_1x_2'^ax_3}= \frac{x_3'^{a+1}+(x_1^a+x_2^a)^a}{x_1x_2x_3'^a}\]
Since $\S$ is a hypersurface, it is a complete intersection, and so it Cohen-Macaulay \cite[Prop. 18.13]{Eis95}, and in particular it is $S2$.\footnote{A ring is Cohen-Macaulay if and only if it satisfies the $Sn$ property for every $n$.}

Let $P$ be a prime ideal in $\S$ containing
\[ \D_\x = <x_1x_2x_3,x_1'x_2x_3,x_1x_2'x_3,x_1x_2x_3'>\]
Since $x_1x_2x_3\subset P$, at least one of $\{x_1,x_2,x_3\}\in P$ by primality.  If any two $x_i,x_j$ are, then
\[ x_k^a=x_ix_jx_kM-x_i^a-x_j^a\in P \Rightarrow x_k\in P\]
If only one $x_i\in P$, then $x_i'x_jx_k\in P$ implies that $x_i'\in P$.  Then $x_i+x_i'^a=x_jx_kM\in P$, which implies $M\in P$.  Additionally, $x_j^a+x_k^a=x_ix_jx_kM-x_i^a\in P$.

Therefore, $P$ contains at least one of the four prime ideals
\begin{equation}\label{eq: components}
 <x_1,x_2,x_3>, <x_1,x_2^a+x_3^a,M>, <x_2,x_1^a+x_3^a,M>, <x_3,x_1^a+x_2^a,M>
\end{equation}
Since $\{x_1,x_2\}$, $\{x_1,M\}$, $\{x_2,M\}$, and $\{x_3,M\}$ are each regular sequences in $S$, it follows that $\codim(\D_\x)\geq2$.  By Lemma \ref{lemma: tccriterion}, $\S=\U$.
\end{proof}
\begin{rem}
The final step of the proof has some interesting geometric content.  In this case, $\D=\D_\x$, and the four prime ideals \eqref{eq: components} are the minimal primes containing $\D$.  Geometrically, they define the irreducible components of $V(\D)$; that is, the complement of the cluster tori.

One of these components ($x_1=x_2=x_3=0$) is an affine line on which every cluster variable vanishes.  The other $3$ components ($x_i=x_j^a+x_k^a=M=0$) are geometrically reducible; over $\mathbb{C}$ they each decompose into $a$-many affine lines.  Over $\mathbb{C}$,  $V(\D)$ consists of $3a+1$-many affine lines, which intersect at the point $x_1=x_2=x_3=M=0$ and nowhere else.
\end{rem}

\subsection{The $(3,3,2)$ cluster algebra}

Consider the initial seed in Figure \ref{fig: 332}.  The cluster algebra $\A_{3,3,2}$ is non-acyclic, by Theorem \ref{thm: nonacyclicrank3}.  Up to permuting the vertices, it is the only non-acyclic $\A_{a,b,c}$ with $0\leq a,b,c\leq 3$ besides $\A_{2,2,2}$ and $\A_{3,3,3}$.

\begin{figure}[h!t]
	\begin{tikzpicture}
	\begin{scope}
		\node (array) at (0,0) {$\sB=\gmat{0 & -3 & 2 \\ 3 & 0 & -3 \\ -2 & 3 & 0 }$};
	\end{scope}
	\begin{scope}[xshift=2.5in]
		\node (label) at (-2,0) {$\Q=$};
		\node[mutable] (x) at (90:1) {$x_1$};
		\node[mutable] (y) at (210:1) {$x_2$};
		\node[mutable] (z) at (-30:1) {$x_3$};
		\draw[-angle 90,relative, out=15,in=165] (x) to (y);
		\draw[-angle 90,relative, out=-15,in=-165] (x) to (y);
		\draw[-angle 90,relative, out=-45,in=-495] (x) to (y);
		\draw[-angle 90,relative, out=15,in=165] (y) to (z);
		\draw[-angle 90,relative, out=-15,in=-165] (y) to (z);
		\draw[-angle 90,relative, out=-45,in=-495] (y) to (z);
		\draw[-angle 90,relative, out=15,in=165] (z) to (x);
		\draw[-angle 90,relative, out=-15,in=-165] (z) to (x);
	\end{scope}
	\end{tikzpicture}
\caption{The exchange matrix and quiver for the $(3,3,2)$ cluster algebra.}
\label{fig: 332}
\end{figure}
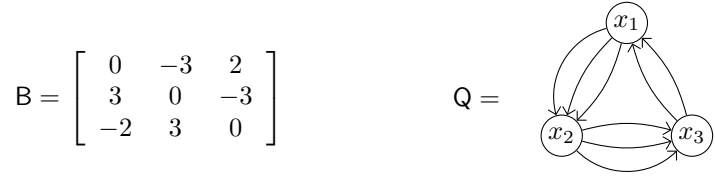

\begin{prop}
The upper cluster algebra $\U_{3,3,2}$ is generated over $\ZZ$ by
\[x_1,x_2,x_3,\]
\[y_0=\frac{x_2^3+x_1^2+x_3^2}{x_1x_3},\;\;\; y_1=\frac{x_1x_2^3+x_2^3x_3+x_1^3+x_3^3}{x_1x_2x_3},\]
\[y_2=\frac{x_2^6+2x_1^2x_2^3+x_1x_2^3x_3+2x_2^3x_3^2+x_1^4+x_1^3x_3+x_1x_3^3+x_3^4}{x_1^2x_2x_3^2},\]
\[y_3=\frac{x_2^{9}+3x_1^2x_2^6+3x_2^6x_3^2+3x_1^4x_2^3+3x_1^2x_2^3x_3^2+3x_2^3x_3^4+x_1^6+2x_1^3x_3^3+x_3^6}{x_1^3x_2^2x_3^3}.\]
The ideal of relations is generated by the elements
\[y_2^2=y_0y_3+2y_3, \;\;\; y_0^2=x_2y_2-y_0+2 \]
\[y_1y_2=x_1y_3+x_3y_3, \;\;\;y_0y_2=x_2y_3+y_2\]
\[y_0y_1=x_1y_2+x_3y_2-2y_1, \;\;\; x_1y_0+x_3y_0=x_2y_1+x_1+x_3\]
\[x_2^2y_2=x_1x_3y_3+3x_2y_0-y_1^2-3x_2, \;\;\; x_2^2y_0=x_1x_3y_2+x_2^2-x_1y_1-x_3y_1\]
\[x_2^3+x_3^2y_0=x_2x_3y_1-x_1^2+x_1x_3.\]
\end{prop}
\begin{proof}
Since $a=3, b=3, c=2,$ and $abc-a^2-b^2-c^2+4=0$, Theorem \ref{thm: nonacyclicrank3} implies that $\A$ is not acyclic.  Thus, Proposition \ref{prop: 3totallycoprime} asserts that $\A$ is totally coprime.  Let $\S$ be the domain in $\K(\A)$ generated by the seven listed elements.  Using Lemma \ref{lemma: presentation} and a computer, we see that the ideal of relations in $\S$ is generated by the elements above.

The following identities imply that $L_\x\subseteq \S$.
\[ x_1'=x_3y_0-x_1,\;\;\;x_2'=x_1x_3y_1-x_1x_2^2-x_2^2x_3,\;\;\;x_3'=-x_3y_0+x_2y_1+x_1\]
The following identities imply that $\S\subseteq \U_\x$.
\[y_0=\frac{x_2^3+x_3^2+x_1'^2}{x_3x_1'} = \frac{(x_1^3+x_3^3)^3+(x_1^2+x_3^2)x_2'^3}{x_1x_3x_2'^3} = \frac{x_1^2+x_2^3+x_3'^2}{x_1x_3'}\]
\begin{eqnarray}y_1 & = & \frac{(x_2^3+x_3^2)^2+x_1'^2(x_2^3+x_3x_1')}{x_2x_3x_1'^2}\nonumber=%\\ & =& 
\frac{(x_1+x_3)^3(x_1^2-x_1x_3+x_3^2)^2+x_2'^3}{x_1x_3x_2'^2}\nonumber\\ &= & \frac{(x_1^2+x_2^3)^2+x_3'^2(x_2^3+x_1x_3')}{x_1x_2x_3'^2} \nonumber
\end{eqnarray}
\begin{eqnarray}y_2 &= &  \frac{(x_2^3+x_3^2)^2+x_3(x_2^3+x_3^2)x_1'+2x_2^3x_1'^2+x_3x_1'^3+x_1'^4}{x_2x_3^2x_1'^2}\nonumber\\ & = & \frac{(x_1^3+x_3^3)^5+(2x_1^2+x_1x_3+2x_3^2)(x_1^3+x_3^2)^2x_2'^3+(x_1+x_3)x_2'^6}{x_1^2x_3^2x_2'^5} \nonumber\\ & = & \frac{(x_1^2+x_2^3-x_1x_3'+x_3'^2)(x_1^2+x_2^3+2x_1x_3'+x_3'^2)}{x_1^2x_2x_3'^2} \nonumber
\end{eqnarray}
\begin{eqnarray}y_3 &= & \frac{(x_2^3+x_3^2-x_3x_1'+x_1'^2)^2(x_2^3+x_3^2+2x_3x_1'+x_1'^2)}{x_2^2x_3^3x_1'^3} \nonumber\\ & = &  \frac{((x_1+x_3)^3(x_1^2-x_1x_3+x_3^2)^2+x_2'^3)^2((x_1+x_3)(x_1^2-x_1x_3+x_3^2)^3+x_2'^3)}{x_1^3x_3^3x_2'^7} \nonumber\\ & = & \frac{(x_1^2+x_2^3-x_1x_3'+x_3'^2)^2(x_1^2+x_2^3+2x_1x_3'+x_3'^2)}{x_1^3x_2^2x_3'^3} \nonumber
\end{eqnarray}

A computer verifies that $(\S x_1x_2x_3:(\S\D_\x)^\infty)=\S x_1x_2x_3$.  By Lemma \ref{lemma: tccriterion}, $\S=\U$.
\end{proof}

\begin{rem}
This example serves of a `proof of concept' for the algorithm of Section \ref{section: algorithm}.  The above generating set has no distinguishing properties known to the authors; it is merely the generating set produced by an implementation of this algorithm.
\end{rem}

\section{Larger examples}\label{section: examples2}

We explicitly present a few other non-acyclic upper cluster algebras.

\subsection{The Markov cluster algebra with principal coefficients}\label{section: MwP}

Consider the initial seed in Figure \ref{fig: MwP}.  As in the previous section, this seed has 3 mutable variables, but it has \emph{principal coefficients} \--- a frozen variable for each mutable variable, and the exchange matrix extended by an identity matrix.  Results about principal coefficients and why they are important can be found in \cite{FZ07}.%  The study of principal coefficients was introduced in \cite{FZ08}, and it remains one of the most important choices of frozen variables.

\begin{figure}[h!t]
	\begin{tikzpicture}
	\begin{scope}
		\node (array) at (0,0) {$\sB=\gmat{0 & -2 & 2 \\ 2 & 0 & -2 \\ -2 & 2 & 0 \\ 1 & 0 & 0 \\ 0 & 1 & 0 \\ 0 & 0 & 1}$};
	\end{scope}
	\begin{scope}[xshift=2.5in,yshift=-.5cm]
		\node (label) at (-2.5,.5) {$\Q=$};
		\node[mutable] (x1) at (90:1) {$x_1$};
		\node[mutable] (x2) at (210:1) {$x_2$};
		\node[mutable] (x3) at (-30:1) {$x_3$};
        \node[frozen] (f1) at (90:2) {$f_1$};
        \node[frozen] (f2) at (210:2) {$f_2$};
        \node[frozen] (f3) at (-30:2) {$f_3$};
		\draw[-angle 90,relative, out=15,in=165] (x1) to (x2);
		\draw[-angle 90,relative, out=-15,in=-165] (x1) to (x2);
		\draw[-angle 90,relative, out=15,in=165] (x2) to (x3);
		\draw[-angle 90,relative, out=-15,in=-165] (x2) to (x3);
		\draw[-angle 90,relative, out=15,in=165] (x3) to (x1);
		\draw[-angle 90,relative, out=-15,in=-165] (x3) to (x1);
        \draw[-angle 90] (x1) to (f1);
        \draw[-angle 90] (x2) to (f2);
        \draw[-angle 90] (x3) to (f3);
	\end{scope}	
	\end{tikzpicture}
\caption{The exchange matrix and quiver for the Markov cluster algebra with principal coefficients.}
\label{fig: MwP}
\end{figure}
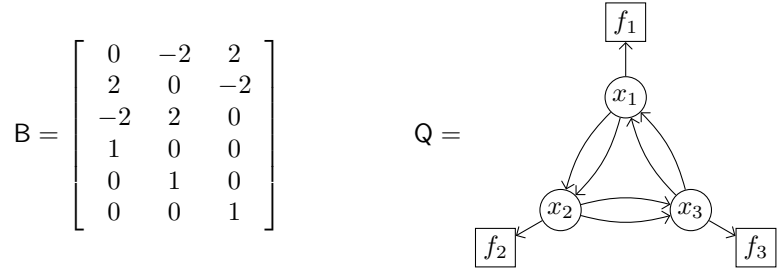

\begin{prop}
The upper cluster algebra $\U$ is generated over $\ZZ[f_1^{\pm1},f_2^{\pm1},f_3^{\pm1}]$ by
\[ x_1,x_2,x_3,\]
\[ L_1=\frac{x_2^2+f_2f_3x_3^2+f_3x_1^2}{x_2x_3}, L_2=\frac{x_3^2+f_3f_1x_1^2+f_1x_2^2}{x_3x_1}, L_3=\frac{x_1^2+f_1f_2x_2^2+f_2x_3^2}{x_1x_2}\]
\[ y_1=\frac{f_1L_1^2+(f_1f_2f_3-1)^2}{x_1}, y_2=\frac{f_2L_2^2+(f_1f_2f_3-1)^2}{x_2}, y_3=\frac{f_3L_3^2+(f_1f_2f_3-1)^2}{x_3}\]
The ideal of relations is generated by the elements
\[ x_1x_2L_3 = x_1^2+f_1f_2x_2^2+f_2x_3^2,\;\;\; y_1y_2L_3 = f_1f_2y_1^2+y_2^2+f_1y_3^2\]
\[ x_2x_3L_1 = x_2^2+f_2f_3x_3^2+f_3x_1^2,\;\;\; y_2y_3L_1 = f_2f_3y_2^2+y_3^2+f_2y_1^2\]
\[ x_3x_1L_2 = x_3^2+f_3f_1x_1^2+f_1x_2^2,\;\;\; y_3y_1L_2 = f_3f_1y_3^2+y_1^2+f_3y_2^2\]
\[ f_3x_1L_3-x_3L_1=\alpha x_2,\;\;\; f_1L_1y_3-L_3y_1=\alpha y_2\]
\[ f_1x_2L_1-x_1L_2=\alpha x_3,\;\;\; f_2L_2y_1-L_1y_2=\alpha y_3\]
\[ f_2x_3L_2-x_2L_3=\alpha x_1,\;\;\; f_3L_3y_2-L_2y_3=\alpha y_1\]
\[ x_1L_2L_3 = f_1f_2x_2L_2+f_1x_1L_1+x_3L_3,\;\;\; y_1L_2L_3 = y_2L_2+f_1y_1L_1+f_1f_3y_3L_3\]
\[ x_2L_3L_1 = f_2f_3x_3L_3+f_2x_2L_2+x_1L_1,\;\;\; y_2L_3L_1 = y_3L_3+f_2y_2L_2+f_2f_1y_1L_1\]
\[ x_3L_1L_2 = f_3f_1x_1L_1+f_3x_3L_3+x_2L_2,\;\;\; y_3L_1L_2 = y_1L_1+f_3y_3L_3+f_3f_2y_2L_2\]
\[ x_2y_3=f_2f_3L_2L_3-\alpha L_1,\;\;\; x_3y_1=f_3f_1L_3L_1-\alpha L_2,\;\;\; x_1y_2=f_1f_2L_1L_2-\alpha L_3\]
\[ x_1y_3 = L_1L_3+f_2\alpha L_2,\;\;\; x_2y_1 = L_2L_1+f_3\alpha L_3,\;\;\; x_3y_2 = L_3L_2+f_1\alpha L_1\]
\[ x_1y_1 = f_1L_1^2+\alpha^2,\;\;\; x_2y_2 = f_2L_2^2+\alpha^2,\;\;\; x_3y_3 = f_3L_3^2+\alpha^2\]
\[ L_1L_2L_3-f_1L_1^2-f_2L_2^2-f_3L_3^2=\alpha^2\]
where $\alpha:=f_1f_2f_3-1$.
\end{prop}
\begin{proof}
The exchange matrix $\sB$ for the initial seed above contains a submatrix that is a scalar multiple of the identity, thus $\sB$ is full rank.  Theorem \ref{thm: fullrank} asserts that $\A$ is totally coprime.  Let $\S$ be the domain in $\K(\A)$ generated by the twelve listed elements.  Using Lemma \ref{lemma: presentation} and a computer, we see that the ideal of relations in $\S$ is generated by the elements above.

The following identities imply that $L_\x\subseteq \S$.
\[ x_1'=x_3L_2-f_3f_1x_1,\;\;\;x_2'=x_1L_3-f_1f_2x_2,\;\;\;x_3'=x_2L_1-f_2f_3x_3\]
The following identities imply that $\S\subseteq \U_\x$.
\[ L_1= \frac{x_1'^2x_2^2+x_1'^2x_3^2f_2f_3+f_3(x_3^2+x_2^2f_1)^2}{x_1'^2x_2x_3} = \frac{x_1^2+x_3^2f_2+f_3x_2'^2}{x_2'x_3} = \frac{x_3'^2+f_2f_3(x_2^2+x_1^2f_3)}{x_2x_3'}\]
\[ L_2= \frac{x_2'^2x_3^2+x_2'^2x_1^2f_3f_1+f_1(x_1^2+x_3^2f_2)^2}{x_2'^2x_3x_1} = \frac{x_2^2+x_1^2f_3+f_1x_3'^2}{x_3'x_1} = \frac{x_1'^2+f_3f_1(x_3^2+x_2^2f_1)}{x_3x_1'}\]
\[ L_3= \frac{x_3'^2x_1^2+x_3'^2x_2^2f_1f_2+f_2(x_2^2+x_1^2f_3)^2}{x_3'^2x_1x_2} = \frac{x_3^2+x_2^2f_1+f_2x_1'^2}{x_1'x_2} = \frac{x_2'^2+f_1f_2(x_1^2+x_3^2f_2)}{x_1x_2'}\]
\begin{eqnarray}y_1 &= &  {\textstyle \frac{ x_1'^4x_2^2+x_1'^4x_3^2f_1f_2^2f_3^2+2x_1'^2(x_3^2+x_2^2f_1)x_2^2f_1f_3+f_1f_3^2(x_3^2+x_2^2f_1)^3+2x_1'^2(x_3^2+x_2^2f_1)x_3^2f_1f_2f_3^2}{x_1'^2x_2^2x_3^2} }\nonumber\\ & = &  \frac{f_1(x_1^2+x_3^2f_2+f_3x_2'^2)^2+(f_1f_2f_3-1)^2x_2'^2x_3^2}{x_1x_2'^2x_3^2} \nonumber\\ & = & \frac{f_1(x_1(x_3'^3+x_3'f_2f_3(x_2^2+x_1^2f_3)))^2+(f_1f_2f_3-1)^2x_1^2x_2^2x_3'^4}{x_1^3x_2^2x_3'^4} \nonumber
\end{eqnarray}
\begin{eqnarray}y_2 &= & {\textstyle  \frac{x_2'^4x_3^2+x_2'^4x_1^2f_2f_3^2f_1^2+2x_2'^2(x_1^2+x_3^2f_2)x_3^2f_2f_1+f_2f_1^2(x_1^2+x_3^2f_2)^3+2x_2'^2(x_1^2+x_3^2f_2)x_1^2f_2f_3f_1^2}{x_2'^2x_3^2x_1^2} }\nonumber\\ & = &  \frac{f_2(x_2^2+x_1^2f_3+f_1x_3'^2)^2+(f_2f_3f_1-1)^2x_3'^2x_1^2}{x_2x_3'^2x_1^2} \nonumber\\ & = & \frac{f_2(x_2(x_1'^3+x_1'f_3f_1(x_3^2+x_2^2f_1)))^2+(f_2f_3f_1-1)^2x_2^2x_3^2x_1'^4}{x_2^3x_3^2x_1'^4} \nonumber
\end{eqnarray}
\begin{eqnarray}y_3 &= & {\textstyle  \frac{x_3'^4x_1^2+x_3'^4x_2^2f_3f_1^2f_2^2+2x_3'^2(x_2^2+x_1^2f_3)x_1^2f_3f_2+f_3f_2^2(x_2^2+x_1^2f_3)^3+2x_3'^2(x_2^2+x_1^2f_3)x_2^2f_3f_1f_2^2}{x_3'^2x_1^2x_2^2} }\nonumber\\ & = &  \frac{f_3(x_3^2+x_2^2f_1+f_2x_1'^2)^2+(f_3f_1f_2-1)^2x_1'^2x_2^2}{x_3x_1'^2x_2^2} \nonumber\\ & = & \frac{f_3(x_3(x_2'^3+x_2'f_1f_2(x_1^2+x_3^2f_2)))^2+(f_3f_1f_2-1)^2x_3^2x_1^2x_2'^4}{x_3^3x_1^2x_2'^4} \nonumber
\end{eqnarray}

A computer verifies that $(\S x_1x_2x_3:(\S\D_\x)^\infty)=\S x_1x_2x_3$.  By Lemma \ref{lemma: tccriterion}, $\S=\U$.
\end{proof}

This presentation is enough to demonstrate an unfortunate pathology of upper cluster algebras.
If $\sB$ is an exchange matrix, and $\sB^\dagger$ is an exchange matrix obtained from $\sB$ by deleting some rows corresponding to frozen variables, then there are natural ring maps
\[ s:\A(\sB)\rightarrow \A(\sB^\dagger),\;\;\; s:\U(\sB)\rightarrow \U(\sB^\dagger)\]
which send the deleted frozen variables to $1$.  It may be naively hoped that the map on upper cluster algebras is a surjection, but this does not always happen.
%For $\sB$ be as in Figure \ref{fig: MwP} and $\sB_{2,2,2}$ as in Figure \ref{fig: rank3}, this products a map
%\[ s:\U(\sB)\rightarrow \U(\sB_{2,2,2})\]
\begin{coro}
For $\sB$ be as in Figure \ref{fig: MwP} and $\sB_{2,2,2}$ as in Figure \ref{fig: rank3}, the map
\[ s:\U(\sB)\rightarrow \U(\sB_{2,2,2})\]
for which $s(x_i)=x_i$ and $s(f_i)=1$ is not a surjection.
\end{coro}

\subsection{The `dreaded torus'}

Consider the initial seed in Figure \ref{fig: DT}.

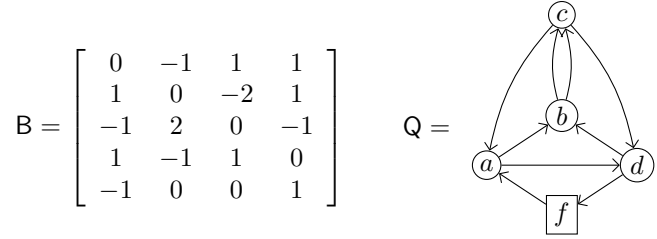
\begin{figure}[h!t]
	\begin{tikzpicture}
	\begin{scope}
		\node (array) at (0,0) {$\sB=\gmat{0 & -1 & 1 & 1 \\ 1 & 0 & -2 & 1 \\ -1 & 2 & 0 & -1 \\ 1 & -1 & 1 & 0 \\ -1 & 0 & 0 & 1}$};
	\end{scope}
\begin{scope}[xshift=2in,yshift=-.5cm]
		\node (label) at (-1.8,.5) {$\Q=$};
		\node[mutable] (a) at (-1,0) {$a$};
		\node[mutable] (b) at (0,.66) {$b$};
		\node[mutable] (c) at (0,2) {$c$};
		\node[mutable] (d) at (1,0) {$d$};
		\node[frozen] (f) at (0,-.66) {$f$};
		\draw[-angle 90] (f) to (a);
		\draw[-angle 90] (a) to (b);
		\draw[-angle 90, out=105,in=255] (b) to (c);
		\draw[-angle 90,relative,out=-15,in=195] (c) to (a);
		\draw[-angle 90] (a) to (d);
		\draw[-angle 90] (d) to (b);
		\draw[-angle 90, out=75,in=285] (b) to (c);
		\draw[-angle 90,relative,out=15,in=165] (c) to (d);
		\draw[-angle 90] (d) to (f);
\end{scope}
	\end{tikzpicture}
\caption{The exchange matrix and quiver for the dreaded torus cluster algebra.}
\label{fig: DT}
\end{figure}

\begin{prop}
The upper cluster algebra $\U$ is generated over $\ZZ[f^{\pm1}]$ by
\[a,b,c,d\]
\[ X=\frac{b^2+c^2+ad}{bc},\;\;\; Y= \frac{ad^2+ac^2+bcf+b^2d}{acd},\;\;\; Z= \frac{a^2d+ac^2+bcf+b^2d}{abd}.\]
The ideal of relations is generated by the elements
\[ bcX=b^2+c^2+ad,\;\;\; cY-bZ=d-a \]
\[ acX-adZ=ab-bd-cf,\;\;\; bdX-adY=cd-ac-bf\]
\[ bXZ-aX-bY-cZ=f.\]
\end{prop}
\begin{proof}
The exchange matrix $\sB$ is full rank, and so Theorem \ref{thm: fullrank} asserts that $\A$ is totally coprime.  Let $\S$ be the domain in $\K(\A)$ generated by the eight listed elements.  Using Lemma \ref{lemma: presentation} and a computer, we see that the ideal of relations in $\S$ is generated by the elements above.

The following identities imply that $L_\x\subseteq \S$.
\[ a'=-cX+dZ+b,\;\;\;b'=cX-b,\;\;\;c'=bX-c,\;\;\;d'=-bX+aY+c\]
The following identities imply that $\S\subseteq \U_\x$.
\begin{eqnarray}X & = & \frac{(b^2+c^2)a'+(bd+cf)d}{a'bc}  = \frac{c^2+ad+b'^2}{cb'}   \nonumber\\
& = &\frac{(c^2+b^2)d'+(ca+bf)a}{d'cb} = \frac{b^2+da+c'^2}{bc'} \nonumber
\end{eqnarray}
\begin{eqnarray}Y & = & \frac{d^2+c^2+a'b}{cd}  =  \frac{b'^2ad^2+b'^2ac^2+b'(c^2+ad)cf+(c^2+ad)^2d}{ab'^2cd}  \nonumber\\
& = &\frac{c'^2d+a(b^2+ad)+c'bf}{ac'd} = \frac{a(ac+bf)+d'^2c+d'b^2}{acd'} \nonumber
\end{eqnarray}
\begin{eqnarray}Z & = & \frac{a^2+b^2+d'c}{ba}  =  \frac{c'^2da^2+c'^2db^2+c'(b^2+da)bf+(b^2+da)^2a}{dc'^2ba}  \nonumber\\
& = &\frac{b'^2a+d(c^2+da)+b'cf}{db'a} = \frac{d(db+cf)+a'^2b+a'c^2}{dba'} \nonumber
\end{eqnarray}

%A computer verifies that $\S\D_\x$ is trivial.  This implies that $(\S abcd:(\S\D_\x)^\infty)=\S abcd$, and so by Lemma \ref{lemma: tccriterion}, $\S=\U$.
A computer verifies that $(\S abcd:(\S\D_\x)^\infty)=\S abcd$.  By Lemma \ref{lemma: tccriterion}, $\S=\U$.
\end{proof}

This presentation makes it easy to explore the geometry of $Spec(\U)$.  One interesting result is the following, which can be proven by computer verification.
\begin{prop}
The induced deep ideal $\U\D$ is trivial.
\end{prop}

As a consequence, $Spec(\U)$ is covered by the cluster tori $\{Spec(\ZZ[x_1^{\pm1},...,x_\nr^{\pm1}])\}$ coming from the clusters of $\U$.  Since affine schemes are always quasi-compact,\footnote{\cite[Exercise 2.13.b]{Har77}.} this cover has a finite subcover; that is, some finite collection of cluster tori cover $Spec(\U)$.

\begin{rem}
This cluster algebra comes from a marked surface with boundary (via the construction of \cite{FST08}); specifically, the torus with one boundary component and a marked point on that boundary component.  In this perspective, the additional generators $X,Y,Z$ correspond to loops.

The epithet `the dreaded torus' was coined by Gregg Musiker in a moment of frustation \--- among cluster algebras of surfaces, it lies in the grey area between having enough marked points to be well-behaved \cite{MulLA,MSW11} and having few enough marked points to be provably badly-behaved (like the Markov cluster algebra).  As a consequence, it is still not clear whether $\A=\U$ in this case (despite the presentation for $\U$ above).
\end{rem}

\section{Acknowledgements}

This paper would not have been possible without helpful insight from S. Fomin, J. Rajchgot, D. Speyer, and K. Smith.  This paper owes its existence to the VIR seminars in Cluster Algebras at LSU, and to the second author's time at MSRI's Thematic program on Cluster Algebras.

\bibliography{MyNewBib}{}
\bibliographystyle{amsalpha}

\end{document}